\documentclass[12pt]{amsart}
\usepackage{amsmath}
\usepackage{amssymb}
\usepackage{graphicx} 
\usepackage{overpic}
\usepackage{enumerate}
\usepackage{wrapfig}
\usepackage{hyperref}
\usepackage[font=small,labelfont=rm]{caption}
\renewcommand{\[}{\left[}

\newcommand{\R}{{\mathbb R}}		%Real numbers
\newcommand{\C}{{\mathcal C}}           %Circle
\newcommand{\cn}{\colon}                %ColoN

\newcommand{\area}{\mathop{\rm Area}}
\newcommand{\lat}{{\mathcal L}}
\newcommand{\Z}{{\mathbb Z}}

\newcommand{\bmat}{\left[\begin{matrix}}
\newcommand{\emat}{\end{matrix}\right]}
\newcommand{\ec}{{\kappa}}               %Euclidean Curvature.
\newcommand{\ac}{{\varkappa}}               %Affine Curvature.
\newcommand{\nn}{\nonumber}	   	%NoNumber
\newcommand{\co}{{\mathop{\rm\bf c}}}
\newcommand{\si}{{\mathop{\rm\bf s}}}
\newcommand{\ol}{\overline}

\renewcommand{\phi}{\varphi}
\numberwithin{equation}{section} 
\newcommand{\bi}[1]{{\bf\itshape #1}}
\renewcommand{\t}{\text{\rm\bf t}}
\newcommand{\n}{\text{\rm\bf n}}
\newcommand{\m}{{\mathfrak m}}

\newtheorem{thm}{Theorem}[section]
\newtheorem{introthm}{Theorem}

\newtheorem{lemma}[thm]{Lemma}
\newtheorem{prop}[thm]{Proposition}
\newtheorem{cor}[thm]{Corollary}

\theoremstyle{definition}
\newtheorem{defn}[thm]{Definition}

\theoremstyle{remark}

\newtheorem{remark}[thm]{Remark}

\newtheorem{example}[thm]{Example}

\catcode`@=11 \@mparswitchfalse  %This puts the \mnote's on the right.

\newcounter{mnotecount}  % Ties the numbering of \mnotes to 
\renewcommand{\themnotecount}{\arabic{mnotecount}} % Set 
\newcommand{\mnote}[1]%  Defines the marginal notes macro.
{\protect{\stepcounter{mnotecount}}$^{\mbox{\footnotesize  $%\!\!\!\!\!\!\,
      \bullet$\themnotecount}}$\marginpar{\parbox[b]{1.2in}{\raggedright\tiny\em
	 \themnotecount:\! #1}} }
\newcommand{\f}{\partial}

\newcommand{\aff}{\Lambda}

\author{Ralph Howard}
\address{Department of Mathematics,
University of South Carolina,
Columbia, SC 29208}
\email{howard@math.sc.edu}
\urladdr{www.math.sc.edu/$\sim$howard}

\title[Bounding lattice points by affine arc length]{Shur type comparison theorems
for affine curves with application to lattice point estimates}

\begin{document}

\begin{abstract} 
    If $c, \ol c\cn [a,b]\to \R^2$ are two convex planar curve parameterized
    by affine  arc length and $A\cn [a,b]\to [0,\infty)$ is the
    area bounded by the restriction $c\big|_{[a,s]}$ and the segment
    between $c(a)$ and $c(s)$ with $\ol A$ the corresponding 
    function for $\ol c$, and the affine curvature are related
    by $\ac(s) \le \ol\ac(s)$, then $A(s)\ge \ol A(s)$.  Also
    for any point of a convex curve we define \emph{adapted
    affine coordinates} centered at the point and give
    sharp estimates on the coordinates of the curve in 
    terms of bounds on the curvature.  Proving these bounds
    involves generalizing classical comparison theorems
    of Strum-Liouville type to higher order and nonhomogenous
    equations.    
    These estimates allow us
    to give sharp bounds on the areas of
    inscribed triangles in terms of affine curvature and
    the affine distance between the vertices.  These inequalities 
    imply upper bounds of the number
    of lattice points on a convex curve in terms of its
    affine arc length.  
\end{abstract}

\maketitle

%%%%%%%%%%%%%%%%%%%%%%%%%%%%%%%%%%%%%%%%%%%%%%%%%%%%%%%%%%%%%%%%%%%%%%
\section{Introduction.}
\label{sec:intro}
%%%%%%%%%%%%%%%%%%%%%%%%%%%%%%%%%%%%%%%%%%%%%%%%%%%%%%%%%%%%%%%%%%%%%%

The results in this paper were motivated first by wanting to give extensions,
or more precisely analogues, of some of the results of the Euclidean
differential geometry of curves to the affine setting;.  The second motivation
was to use these results to give upper bounds on the number of lattice points
on a convex curve in terms of the affine arc length and bounds on affine
curvature of the curve.

Let $\C$ be an embedded connected curve in the plane.  Then we call $\C$ a
\bi{convex} if and only if $\C$ is on the boundary of its convex hull.  The
curve is a \bi{closed convex} curve if and only if it is the boundary of a
bounded convex set.  

A basic result in the Euclidean differential geometry of curves is the
comparison theorem of A.~Schur: if $\C$ and $\ol \C$ are two convex non-closed
curves of the same length and the curvature of $\ol \C$ is pointwise greater
than the curvature of $\C$, then the distance between the endpoints of $\ol \C$
is smaller than the distance between the endpoints of $\C$. That is the greater
the curvature the smaller the distance between the endpoints.  (See \cite[Page
36]{Chern-curves} or \cite[Thm.~2-19, Page 31]{Gugg}  for the precise statement
and a proof.)  The original proof is in \cite{Schur-curve}.  Some Some
extensions and generalization are given in \cite{Sullivan-tot-curv} (higher
dimensions and minimal regularity), \cite{Epstein-Schur} (version in hyperbolic
space) and \cite{Lopez-Schur} (version in the Lorentzian plane).

In affine geometry the distance between points in the plane is not defined, but
the area bounded by the segment joining two points on the curve and the curve
(see Figure \ref{fig:intro1}) is defined.  Recall a \bi{special affine motion}
of $\R^2$ is a map of the form $v \mapsto Mv+b$ where $M$ is a linear map with
$\det(M)=1$ and $b\in \R^2$.  (See Section~\ref{sec:prelims} for our
conventions and definitions
of affine arc length and affine
curvature.) The affine version of Shur's Theorem is the greater curvature
the smaller the bounded area:

\begin{introthm}\label{introthmA}
	Let $c, \ol c\cn [a,b]\to \R^2$ be  curves parameterized 
	by affine arc length with affine curvatures $\ac$ and $\ol \ac$
	respectively
	and let $A$ and $\ol A$ be the respective areas bounded by the curves 
	and the segment between their endpoints as in Figure \ref{fig:intro1}
	(see Section~\ref{sec:prelims} for the precise definition of $A$ and $\ol A$).
	Assume that $\ol \ac$ satisfies any one of the following three
	conditions (a) $\ol \ac$ is constant, (b) $\ol\ac \le 0$, or (c)
	$\ol \ac\le k_1$ where $k_1$ is a positive constant with
	$k_1\le (\pi/(b-a))^2$.  Then
	\begin{align*}
		\ac \le \ol \ac \text{ on } [a,b] \quad &\text{implies} \quad A\ge \ol A\\
		\ac \ge \ol \ac \text{ on } [a,b] \quad &\text{implies} \quad A\le \ol A
	\end{align*}
	In either case $A=\ol{A}$ implies $c$ is the image of $\ol c$ by
	a special affine motion.  
\end{introthm}

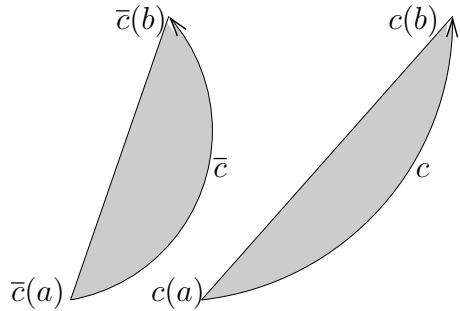
\begin{figure}[ht]
\begin{overpic}[height=1.5in]{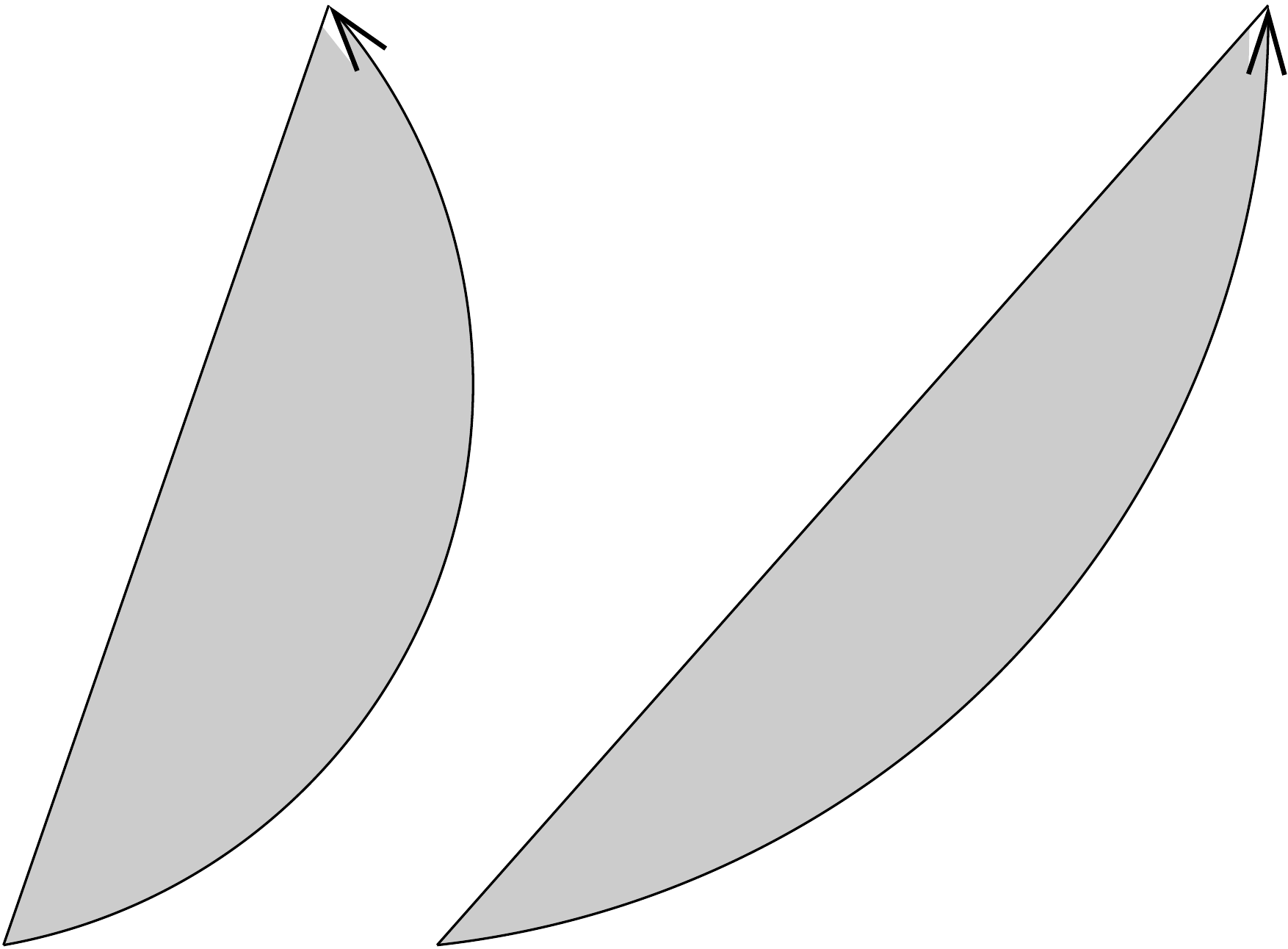}
\put(-15,0){$\ol c(a)$}
\put(21,0){$c(a)$}
\put(82,70){$c(b)$}
\put(12,70){$\ol c(b)$}
\put(37,32){$\ol c$}
\put(89,32){$c$}
\end{overpic}
\caption{The curves $c$ and $\ol c$ showing the areas bounded by
the curves and the secant segments between their endpoints.}
\label{fig:intro1}
\end{figure}

This is a special case of Theorem \ref{thm:sec-comp} below.
Besides this result we give comparison results for the areas of 
inscribed triangles (Proposition \ref{prop:inside_tri} 
Theorem \ref{thm:triangle2}) 
and coordinate systems centered at point on the curve that are
``adapted'' to the geometry of the curve (Theorem~\ref{thm:coord_bds}).

These results are based on generalizations of Sturm type comparison theorems
for homogeneous second order differential differential equations to equations
that are non-homogeneous and of higher order.  To give an example of the
relevance of such results, in Theorem \ref{sec:intro} let $A=A(b)$ be  viewed as
a function of the end point $c(b)$, then a special case of Theorem
\ref{thm:A_ode}, is that $A$ satisfies the third order initial value problem 
$$
A''' + \ac A' = \frac12, \qquad A(a)=A'(a) = A''(a)=0.
$$
Thus Theorem \ref{introthmA} follows from a comparison theorem
for differential equations of this form.  Section \ref{sec:initial_valus_problems} 
has a general theory for comparisons of solutions to linear ordinary 
differential equations.

Finally we consider the problem of bounding the number of
lattice points on a convex curve by bounds its affine arc length
and bounds on its affine curvature.    The original result about the number of
lattice points on a convex curve is result of Jarn\'{\i}k \cite{Jar} that
a closed strictly convex curve of length $L$ contains at
most $3(2\pi)^{-1/3}L^{2/3} + O(L^{1/3})$ and the exponent and constant
of the leading term are best possible.  There have been improvements 
of Jarn\'{\i}k's result which involve bounds of the higher derivatives of
the curve by Bombieri and Pila\cite{BomP} and Swinnerton-Dyer \cite{SwD}
and others.  In the case of bounds
on the fourth derivative it is natural to interpret the bounds
on the derivatives in terms of affine curvature
and to give the bounds  in terms of
the affine arc length.
In the case of arcs on ellipses this was done
in \cite{HowTri}.
Anther reason affine arc length more natural
than Euclidean arc length
for these questions is that the lattice $\Z^2$ in variant under
the group $\operatorname{SL}{2,\Z}$ of $2\times 2$ integer 
matrices of determent one.   Affine arc length is invariant 
under this group, while Euclidean arc length is not.

\begin{introthm}
Let $k_0$ and $k_1$ be constants with $k_0\le k_1$.  
Then there is a number $L_{k_0}>0$ and so that if
$k_1\le (\pi/(2L_{k_0}))^2$ (which is automatic if
$k_1\le 0$) so that if $\C$ is a convex curve 
whose affine curvature satisfies $k_0\le \ac\le k_1$
and $m = \lfloor \Lambda(\C)/(2L_{k_0})\rfloor$ 
then
$$
    \#(\Z^2\cap \C)\le 2m+2.
$$
\end{introthm}

When $k_0=0$ the constant is $L_0= 1$. In this case 
the example of  the parabolic arc $\C=\{ (s,s(s-1)/2): 0\le s \le 2m+1\} $
shows this estimate is sharp.   For $k_0<0$ the number $L=L_{k_0}$ is
the solution to 
$$
\bigg( \frac{\sinh(\sqrt{|k_0|}\,L)}{\sqrt{|k_0|}}\bigg)
\bigg( \frac{\cosh(\sqrt{|k_0|}\, L)-1}{|k_0|}        \bigg) = \frac12.
$$
In this case there are also examples where the bound is sharp,
see Section~\ref{sec:examples}. There is a similar equation
for $k_0$ when $k_0>0$.  However the  discussion and examples in 
Paragraph \ref{subsec:k>0} show these results are 
of greatest interest when $k_0\le 0$.

%%%%%%%%%%%%%%%%%%%%%%%%%%%%%%%%%%%%%%%%%%%%%%%%%%%%%%%%%%%%%%%%%%%%%%
\section{Affine arc length, affine curvature, and the
differential equation for the area function.}
\label{sec:prelims}
%%%%%%%%%%%%%%%%%%%%%%%%%%%%%%%%%%%%%%%%%%%%%%%%%%%%%%%%%%%%%%%%%%%%%%

We give our conventions on the geometry of affine curves.
Other sources for this material are \cite{Blaschke},
\cite{Gugg},  and \cite{Spivak2}.
If $v,w\in \R^2$ then $v\wedge w = v_1w_2-v_2w_1$ where
$v=(v_1,v_2)$ and $w=(w_1,w_2)$.  This is the determinant
of the matrix with rows $v$ and $w$.
If $I$ is an interval and $\gamma\cn I \to \R^2$ is a $C^2$
curve, where the velocity vector $\gamma'(t)$ and the acceleration
vector $\gamma''(t)$ are linearly independent for all $t\in I$,
then a direct calculation shows the one form
$$
\left( \gamma'(t)\wedge \gamma''(t)\right)^{1/3}\,dt
$$
is invariant under both $C^2$ reparameterizations and special affine
motions of $\R^2$.  The \bi{affine arc length} of $c$, denoted
$\Lambda(c)$, is 
defined by integrating this one form. Thus $\Lambda(c) = \int_a^b
(c'(t)\wedge c''(t))^{1/3}\,dt$ where $[a,b]$ is
the domain of $c$.  Equivalently if $c\cn I \to \R^2$,
then $s$ is affine arc length along $c$ if and only if 
$$
c'(s)\wedge c''(s)\equiv 1.
$$
This gives the curve has a natural orientation: the 
orientation for which the basis $c'(s)$, $c''(s)$ is always positive
(that is a right handed) along $c$.  This is the same as
the orientation that makes the one form $c'(t)\wedge c''(t)^\frac13 \,dt$
positive.
The vector $c'(s)$ is
the \bi{affine tangent vector} and $c''(s)$ is the \bi{affine normal
vector}.  
With this orientation the
Euclidean curvature, $\ec$, is positive which 
implies $c$ is locally convex
in the sense that at any point its tangent line is a locally a
support line in that near the point the curve lines in the closed
half plane bounded by the tangent line and with the affine
normal pointing into the half plane.

Taking the derivative of $c'(s)\wedge c''(s)=1$ and using
$c'(s)\wedge c'(s)=0$ gives
$$
c'(s)\wedge c'''(s)=0.
$$
This implies $c'''(s)$ is linearly dependent on $c'(s)$ and therefore
for some scalar function $\ac(s)$,
$$
c'''(s) = -\ac(s) c'(s).
$$
The function $\ac$ is the \bi{affine curvature} of $c$. (This sign is chosen
so that ellipses have constant positive affine curvature and hyperbolas 
have constant negative curvature.)  The affine curvature determines
a curve up to an affine motion.  

\begin{thm}\label{thm:unq} Let $I$ be in interval in $\R$ and
$c_1,c_2\cn I\to \R^2$ be $C^3$ affine unit speed curves that have the 
same affine curvature at each point.  Then $c_1$ and $c_2$
differ by an affine motion.  That is for some linear map $M\cn \R^2\to \R^2$
with $\det(M)=1$ and some $b\in \R^2$ where holds $c_2(s) = Mc_1(s)+b$
for all $s\in I$.\qed
\end{thm}

Proofs can be found in \cite{Blaschke}, \cite{Gugg}, and \cite{Spivak2}.
It is worth remarking that while this theorem only requires $c_1(s)$
and $c_2(s)$ to be $C^3$ functions of the affine arc length,
the images of these curves will be $C^4$ immersed submanifold 
of $\R^2$.  See Proposition \ref{prop:C4}
\medskip

If $c\cn I \to \R^2$ for some interval $I$, $a\in I$ and  $p_0\in \R^2$, 
define 
$$
A_{c,a,p_0}(s):= \begin{cases}
	\text{Signed area bounded by  the restriction $c\big|_{[a,s]}$ }&\\[8pt]
	\text{and the segments $\overline{p_0c(a)}$
	and $\overline{p_0c(s)}$.}
\end{cases}
$$
See Figure \ref{fig:cones}.  To give precise and more computationally
useful definition let $f$ be the function
$$
f(t,s) = (1-t)p_0 + t c(s), \qquad 0\le t\le 1 , \quad s\in I.
$$
Then $A_{c,a,p_0}(s)$ is the signed area of the set of  points
$ \{f(t,\sigma): 0\le t\le 1 \text{ and $\sigma$ between $a$ and $s$}\}  $,
which is a quantity we can compute with an integral.
The partial derivatives and Jacobian of $f$ are
\begin{align*}
\frac{\partial f}{\partial t}  &= c(s) - p_0\\
\frac{\partial f}{\partial s}  &= t c'(s)\\
 \frac{\partial f}{\partial t}  \wedge\frac{\partial f}{\partial s}  &=t (c(s)-p_0) \wedge c'(s). 
\end{align*}
So a precise definition of  $A_{c,a,p_0}$ is
\begin{align}\label{A=int_f}
A_{c,a,p_0}(s)&= \int_a^s \int_0^1 \frac{\partial f}{\partial t}(t,\sigma)  \wedge\frac{\partial f}{\partial s}(t,\sigma) \,dt\,d\sigma\\\nn
&= \int_a^s \int_0^1 t (c(\sigma)-p_0) \wedge c'(\sigma)\,dt\,d\sigma\\\nn
&=\frac12 \int_a^s (c(\sigma)-p_0) \wedge c'(\sigma)\,d\sigma.
\end{align}
In computing the area $A_{c,a,p_0}(s)$ the points where 
$f_t \wedge f_u =t (c(\sigma)-p_0) \wedge c'(\sigma)$ is positive 
the area is counted as positive, and when this is negative
the area is negative the area is negative.

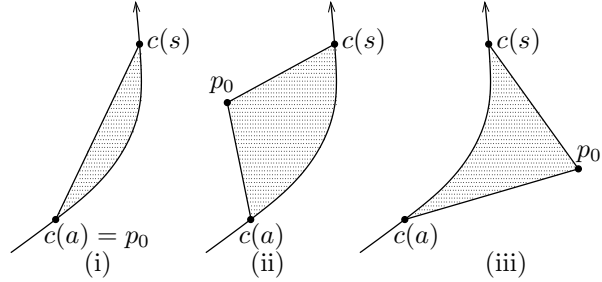
\begin{figure}[hb]\footnotesize
	\begin{overpic}[width=3in]{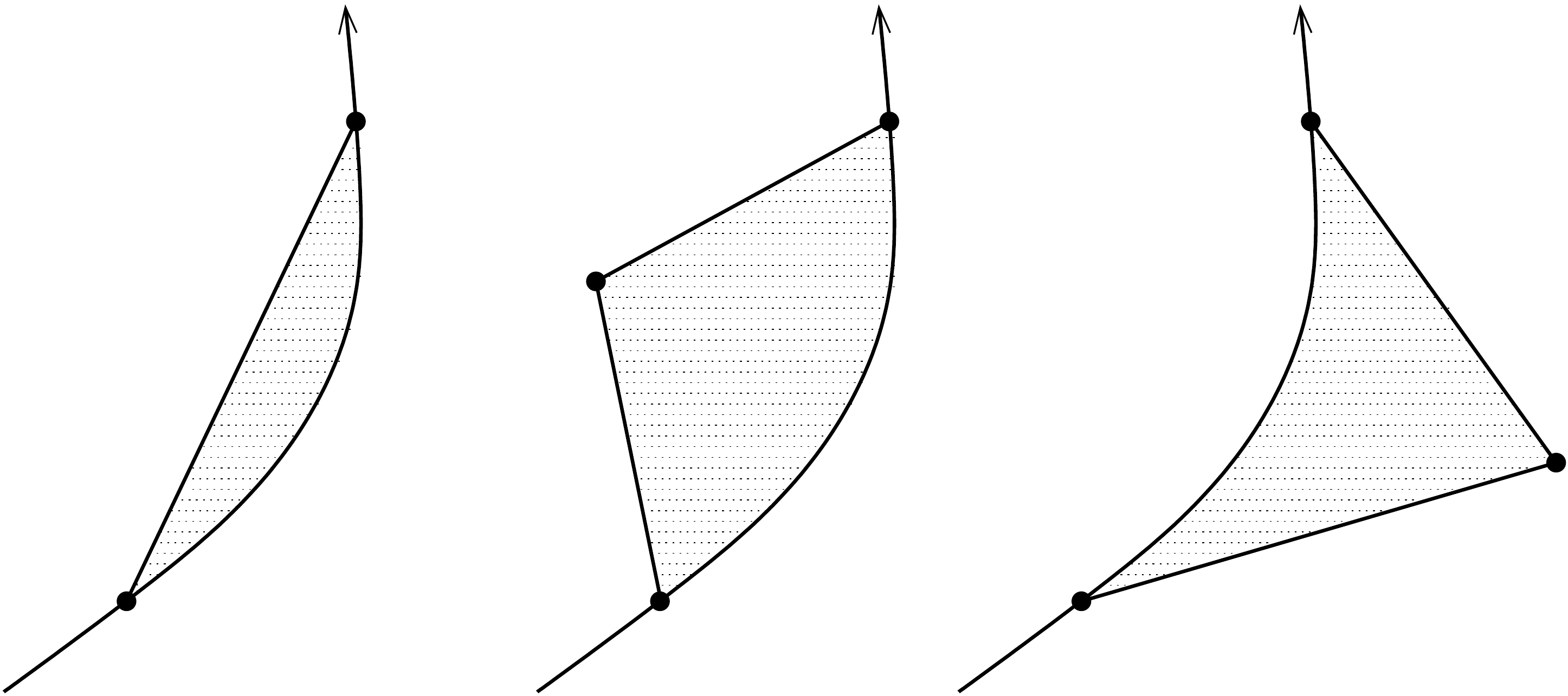}
		\put(6,2){$c(a)=p_0$}
		\put(40,2){$c(a)$}
		\put(67,2){$c(a)$}
		\put(24,36){$c(s)$}
		\put(58,36){$c(s)$}
		\put(85,36){$c(s)$}
		\put(99,17){$p_0$}
		\put(35,29){$p_0$}
		\put(13,-3){(i)}
		\put(42,-3){(ii)}
		\put(83,-3){(iii)}
	\end{overpic}
	\caption{The area $A_{c,a,p_0}(s)$ for some choices of the point $p_0$.
	In (i) and (ii) we have $A_{c,a,p_0}(s)>0$ and in (iii)  
	$A_{c,a,p_0}(s)<0$.}
	\label{fig:cones}
\end{figure}

By taking the first three derivatives of $A(s):=A_{c,a,p_0}(s)$ and
using $c'(s)\wedge c''(s)=1$ and $c'''(s) = -\ac(s)c(s)$
we find it 
satisfies a third order differential equation.
\begin{align}\label{A'}
	A'(s) &= \frac12 (c(s)-p_0) \wedge c'(s)\\ \label{A''}
A''(s) &= \frac12 c'(s)\wedge c'(s) + \frac12 (c(s)-p_0)\wedge c''(s)\\ \nonumber
&= \frac12 (c(s)-p_0)\wedge c''(s)\\ \label{A'''}
A'''(s) &= \frac12 c'(s)\wedge c''(s) + \frac12 (c(s)-p_0) \wedge c'''(s)\\ \nonumber
& = \frac12 + \frac12 (c(s)-p_0) \wedge (-\ac(s) c'(s))\\ \nonumber
&= \frac12 - \ac(s)A'(s).
\end{align}
This proves:

\begin{thm}\label{thm:A_ode}
For any point $p_0$ the area 
function $A(s)=A_{c,a,p_0}(s)$ satisfies the third order differential equation
$$
A''' + \ac A' = \frac12.
$$
with initial conditions 
$$
A(a)=0,\quad A'(a) = \frac{1}{2} (c(a)-p_0)\wedge c'(a), \quad A''(a) = \frac12 (c(a)-p_0)\wedge c''(a).
$$
\qed
\end{thm}

%%%%%%%%%%%%%%%%%%%%%%%%%%%%%%%%%%%%%%%%
\subsection{Some notation and conventions.}
\label{sub:notation}
%%%%%%%%%%%%%%%%%%%%%%%%%%%%%%%%%%%%%%%%

Some of our results are awkward to state for curves given with
a particular parameterization on an interval, that is as $c \cn I \to \R^2$.
So we will use $\C$ and subscripted variants to denote an embedded 
convex curve.  If $p\in \C$ then $\t_\C(p)$ and $\n_\C(p)$ will be 
the affine tangent and affine normal.  Explicitly if $c \cn I \to \C$ 
is a local affine unit speed local parameterization of $\C$ with $c(s_0)=p$,
then $\t_\C(p)=c'(s_0)$ and $\n_\C(s_0)=c''(s_0)$.    The affine curvature of 
$\C$ at $p$ is denoted $\ac(p)$.

%%%%%%%%%%%%%%%%%%%%%%%%%%%%%%%%%%%%%%%%%%%%%%%%%%%%%%%%%%%%%%%%%%%%%%
\section{Comparison results for initial value problems.}
\label{sec:initial_valus_problems}
%%%%%%%%%%%%%%%%%%%%%%%%%%%%%%%%%%%%%%%%%%%%%%%%%%%%%%%%%%%%%%%%%%%%%%

\begin{defn}\label{LKernel}
Let $I$ be an interval in $\R$ and let $\mathcal D$ be a linear differential
operator defined on $C^n(I)$ by 
$$
\mathcal D y = \frac{d^ny}{ds^n}  + \sum_{j=0}^{n-1} a_j(s) \frac{d^jy}{ds^j}
$$
where the functions $a_0 ,\ldots,  a_{n-1}$ are continuous on $I$. 
Then the \bi{Lagrange kernel} for $\mathcal D$ on this interval
is the function $K\cn I \times I \to \R$ defined by the initial 
value problem 
\begin{equation}\label{eq:K_eqn}
\frac{\partial^n K}{\partial s^n}(s;r) + 
\sum_{j=0}^{n-1} a_j(s) \frac{\partial^j K}{\partial s^j}(s;r) =0
\end{equation}
\begin{equation}\label{eq:K_initial}
\frac{\partial^j K}{\partial s^j} (r;r) =0 \quad \text{for} \quad 0\le j \le n-2,\qquad
\frac{\partial^{n-1} K}{\partial s^{n-1}} (r;r)=1.
\end{equation}
\end{defn}

Holding $r$ fixed the equation \eqref{eq:K_initial} is a 
linear ordinary differential equation for the function 
$s \mapsto K(s;r)$ so the existence and uniqueness of
$K$ follows from the existence and uniqueness theorem
for linear ordinary differential equations.  The continuity 
of $K(s;r)$ follows from results on the continuous dependence
of solutions of differential equations on initial conditions
and parameters 
(cf.\ \cite[Chapter~1]{CL}).

\begin{remark}
There does not seem to be a standard name for this kernel.  Some 
authors refer to it as a Green's function (e.g.\ \cite{Linear}), but as this term is 
usually reserved for boundary value problems rather than initial
value problems this seems a little misleading.  As it is used
to solve the Cauchy problem for the non-homogeneous  $\mathcal D y=f$ 
some authors (e.g.~\cite{Linear-ODE}) refer to it as the Cauchy kernel.
But as it is just the kernel one gets by applying Lagrange's method
of variation of parameters 
(cf.\ \cite[Pages 145--147]{Linear}) it seems likely that Lagrange was the
first to write down some form of it.
\end{remark}

\begin{thm}\label{thm:LK-solves}
Let $\mathcal D$ be as in Definition \ref{LKernel}, the function 
$f\cn I \to \R$ continuous, and $r\in I$.  Then, denoting the $j$-th
derivative of $y$ by $y^{(j)}$, 
the solution to the initial value problem
$$
\mathcal D y=f, \qquad y^{(j)}(r)=0 \quad\text{for}\quad j=0,1 ,\ldots, n-1
$$
is 
$$
y(s) = \int_r^s K(s;t)f(t)\,dt.
$$
\end{thm}

\begin{proof}
This is known (and seems to be somewhat of a folk theorem among some applied
mathematicians) and a proof of a somewhat 
stronger version can be found in \cite[Theorem 4-2, Page 149]{Linear}. 
We include a short proof for completeness.  From the fundamental theorem
of calculus and Leibniz's 
formula for differentiating the integral of a function depending on a parameter 
\begin{equation}\label{eq:Lib}
	\frac{\partial }{\partial s} \int_r^s \frac{\partial^j K}{\partial s^j} (s;t)f(t)\,dt
	= \frac{\partial^jK}{\partial s^j} (s;s)f(s) + \int_r^s \frac{\partial^{j+1} K}{\partial s^{j+1}} (s;t)f(t)\,dt
\end{equation}
holds for $j=0 ,\ldots, (n-1)$.  When $j\le n-2$ we have 
$({\partial^jK}/{\partial s^j}) (s;s)=0$ and this gives
$$
\frac{\partial }{\partial s} \int_r^s \frac{\partial^j K}{\partial s^j} (s;t)f(t)\,dt
= \int_r^s \frac{\partial^{j+1} K}{\partial s^{j+1}} (s;t)f(t)\,dt
$$
and therefore for $j= 0,1 ,\ldots, n-1$ 
$$
y^{(j)}(s) = \int_r^s \frac{\partial^j K}{\partial s^j} (s;t)f(t)\,dt.
$$
When $j=n-1$, using $ ({\partial^{n-1}K}/{\partial s^{n-1}}) (s;s)=1$)
and the differential equation for~$K$,
\begin{align*}
    y^{(n)}(s)&= f(s)+ \int_r^s \frac{\partial^n K}{\partial s^n}(s;t)f(t)\,dt\\
          &= f(s)+ \int_r^s\left( -\sum_{j=0}^{n-1} a_j(s) \frac{\partial^j K}{\partial s^j}(s;t)  \right)f(t)\,dt\\
&=f(s) - \sum_{j=0}^{n-1} a_j(s) \int_r^s \frac{\partial^j K}{\partial s^j}(s;t)f(t)\,dt\\
&= f(s) - \sum_{j=0}^{n-1} a_j(s) y^{(j)}(s) 
\end{align*}
and therefore $\mathcal D y=f$.  The initial conditions $y^{(j)}(r)=0$
for $j=0,1, ,\ldots, n-1$ follow from $\int_r^r=0$.
\end{proof}

\begin{defn}\label{K_is_pos}
	Let $I$ be an interval and $\mathcal D$ a linear differential
	operator as in Definition \ref{LKernel} and let $K$ be
	the Lagrange kernel of $\mathcal D$.  Then $K$
	is \bi{forward positive} on $I$ if and only if for all $r,s\in I$
	$$
	s>r \quad \text{implies}\quad K(s;r)\ge 0.
	$$
\end{defn}

\begin{remark}
If $K$ is forward positive in this sense, then
for a fixed $r$ the function given by $y(s):=K(s;r)$ is a
not identically zero solution to a homogeneous linear differential equation.
The zeros of a such a solution are isolated.  Thus for fixed $r$, the 
zero set of $s \mapsto K(s;r)$ is a discrete set.  Therefore
when $K$ is forward positive $s \mapsto K(s;r)$ is positive almost
everywhere on $I\cap [r,\infty)$.  This fact will be used several times. 
\end{remark}

\begin{remark}
For any linear differential operator $\mathcal D$ is in Definition
\ref{LKernel} defined on an interval $I$ and $r\in I$
the initial conditions \eqref{eq:K_initial} there is a 
$\delta>0$ so that $K(s;r)>0$ for $r<s<r+\delta$.  Using this, continuity,
and a compactness argument, it follows that for any $r\in I$ there is
an interval $I_0$ with $r\in I_0 \subseteq I$ so that the restriction
$K\big|_{I_0\times I_0}$ is forward positive on $I_0$.    
\end{remark}

The definition of a  Lagrange kernel being forward positive is motivated  
by the following comparison result.  (This result can be be realized 
to a larger class of differential operators, but the statement
of the results are awkward to state and more generality is not
needed here.)

\begin{thm}\label{thm:comp-y'>0}
	Let $\ac, \ol\ac \cn [a,b]\to \R$ be continuous and $n$ and $\ell$
	integers with $0\le \ell < n$.  Let $y,\ol y\cn [a,b]\to \R$
	be $C^n$ functions that satisfy the $n$-order differential 
	equations 
	\begin{align}\label{eqn:y}
		y^{(n)} + \ac y^{(\ell)}=f \quad \text{and} \quad
		\ol y^{(n)} + \ol \ac\,  \ol y^{(\ell)}   =f
	\end{align}
	for some continuous function $f$ and have the same initial
	conditions at $a$:
	$y^{(j)}(a)=\ol y^{(j)}(a)$ for $j=0,1 ,\ldots, n-1$.	Assume
	\begin{enumerate}[(a)]
	\item 	The Lagrange kernel,  $\ol K$, of the differential
		operator $\ol y \mapsto \ol y^{(n)} + \ol\ac\,\ol y^{(\ell)}$ is forward positive 
		on $[a,b]$, and
	\item $y^{(\ell)}>0$ almost everywhere on $[a,b]$.  
	\end{enumerate}
	Then 
	\begin{align}\label{+3}
\ac \le \ol\ac \text{ on } [a,b]\quad \text{implies}\quad y\ge \ol y \text{ on } [a,b],\\
	\ac \ge \ol\ac \text{ on } [a,b]\quad \text{implies}\quad y\le \ol y \text{ on } [a,b].
	\label{+4}
	\end{align}
	Moreover in either of these cases if
	$y(b)=\ol y(b)$, then  $\ac(s)=\ol \ac(s)$
	and $y(s)=\ol y(s)$ for all $s\in [a,b]$.
\end{thm}

\begin{proof}
	Subtract the first equation in \eqref{eqn:y} form the second to
	and rearrange a bit to get
	$$
	(\ol y - y)^{(n)} + \ol\ac ( y - y)^{(\ell)} = -(\ol \ac - \ac) y^{(\ell)}
	$$
	As $y$ and $\ol y$ have the same initial conditions at $a$ the
	function 
	$\ol y - y$ and its first $(n-1)$ derivatives vanish at $a$. 
	Therefore Theorem \ref{thm:LK-solves} yields
	$$
	\ol y(s) - y(s) = - \int_a^s \ol K (s;t) (\ol \ac(t)-\ac(t))y^{(\ell)}(t)\,dt.
	$$
	By our assumptions for each $s$ the inequality $ K (s;t) y^{(\ell)}(t)>0$
	holds for almost all $t\in [s,b]$.
	Thus if $\ac \le \ol \ac$ on $[a,b]$ we have $y \ge \ol y$ on  $[a,b]$,
	and if $y(b) = \ol y(b)$, then $\ol \ac(t) = \ac(t)$ for $t\in [a,b]$
	which implies $y$ and $\ol y$ stratify the same initial value problem
	on $[a,b]$ and thus are equal on this interval.  A similar argument holds
	if $\ac \ge \ol \ac$ on $[a,b]$.
\end{proof}

To give the Lagrange kernel for the constant coefficient linear operators
related to our problem 
it is convenient to introduce some notation.  
Let $k$ be a real number and define
functions $\co_k$ and $\si_k$  by the initial value problems
\begin{align*}
\co''_k + k\co_k=0, \qquad \co_k'(0)=1, \quad \co_k'(0)=0\\
\si_k''+k\si_k=0,\qquad \si_k(0)=0, \quad \si_k'(0)=1
\end{align*}
or more explicitly 
$$
\co_k(s)= \begin{cases}
\cos(\sqrt k\, s),& k>0;\\
1,& k=0;\\
\cosh(\sqrt{|k|}\, s),& k<0.
\end{cases}
\quad \si_k(s) = \begin{cases}
	\dfrac{\sin(\sqrt {k}\, s)}{\sqrt k},& k>0;\\
	s,& k=0;\\
	\dfrac{\sinh(\sqrt{|k|}\,s)}{\sqrt{|k|}},&k<0.
\end{cases}
$$
These satisfy
\begin{align}
	\co_k'&=-k\si_k,  \qquad \si'=\co_k \label{trig_ders}\\
	\co_k^2 + k \si_k^2&=1 \label{trig_py}\\
	\co_k(a+s) &= \co_k(a) \co_k(s) - k \si_k(a) \si_k(s)\label{add_co}\\
	\si_k(a+s)&= \si_k(a)\co_k(s) + \co_k(a)\si_k(s)\label{add_si}.
\end{align}
Possibly the easiest way to see the derivative  formulas hold is to note $\co_k'$
and $-k\si_k$ are both solutions to the initial value problem
$y''+ky=0$, $y(0)=0$, $y'(0)=-k$, and $\si_k'$ and $\co_k$
are solutions to $y''+ky=0$, $y(0)=1$ and $y'(0)=0$. 
These imply  $\co_k^2 + k \si_k^2$
has zero as its derivative and therefore is constant, this implies
\eqref{trig_py} holds.  For the addition formula for $\co_k$,
note the left and right hand sides of \eqref{add_co} 
both satisfy the initial value problem $u''+ku=0$,
$u(0) = \co_k(a)$ and $u'(0) = \co_k'(a) = -k \si_k(a)$.
A similar argument shows the addition formula for $\si_k$ holds.

\begin{prop}\label{prop:cont-kernels}
Let $k\in \R$ and $I$ an interval. 
\begin{enumerate}[(a)]
\item The  Lagrange kernel for $y\mapsto y''+ ky$ on $I$ is
$$
P_k(s;r) = \si_k(s-r).
$$
When $k\le 0$, this is forward positive on all intervals $I$.
When $k>0$ this is forward positive on all intervals of length $L$ 
satisfying $k\le (\pi/L)^2$.
\item The Lagrange kernel for $y \mapsto y'''+ky'$ is 
$$
Q_k(s;r) = \begin{cases}
	\dfrac{1-\co_k(s-r)}{k}, & k\ne 0;\\[8pt]
\dfrac{(s-r)^2}{2},& k=0. 
\end{cases}
$$
and this is forward positive on all intervals for all $k$.
\end{enumerate}
\end{prop}

\begin{proof}
A direct calculation using \eqref{trig_ders} shows these functions satisfy 
the conditions defining the Lagrange kernel and it is straightforward
to check when they are forward positive.
\end{proof}

\begin{lemma}\label{lem:Sturm}
	Let $\ac\cn [a,b]\to \R$ be continuous and let
	$u\in C^2([a,b])$ satisfy 
	$u''+\ac u=0$.
	\begin{enumerate}[(a)]
	\item If $u(a) = 0$, $u'(a)\ne 0$, and $k_0\le (\pi/(b-a))^2$, then
	$u\ne 0$ on $(a,b)$.
\item If $u(a)\ne 0$, $u'(a) =0$, and $k_0\le (\pi/(2(b-a))^2$, 
	then $u\ne 0$ on $(a,b)$.  
	\end{enumerate}
\end{lemma}

\begin{proof}
While (a) can be proven using the Sturm Comparison Theorem,  \cite[Theorem 1.1 Page 208]{CL},
we give a short proof based on Theorem \ref{thm:comp-y'>0} which has the advantage
of working for (b) as well.
Without loss of generality we may assume $u'(a)=1$, in which case
$u(s)>0$ for $s>a$ and $s$ near $a$.  Towards a contradiction assume
$u$ has a zero in $(a,b)$ and let $b^*$ be the smallest 
zero of $u$ in $(a,b)$.  Then $u>0$ on $(a,b^*)$.  
Let $\ol u(s) := \si_{k_0}(s-a) $.  Then 
$$
\ol u'' + k_0 \ol u =0, \qquad \ol u(a) =0,\quad \ol u'(a)=1.
$$
By Proposition \ref{prop:cont-kernels} the Lagrange kernel
for $\ol u \mapsto \ol u'' + k_0 \ol u$ is $\ol K(s;r)= P_{k_0}(s;r)=\si_{k_0}(s-r)$
and this is forward positive on all intervals (when $k_0\le0$) or
on any interval whose length satisfies $k_0\le (\pi/L)^2$ (when $k_0>0$).
Therefore by Theorem \ref{thm:comp-y'>0} (with $n=2$ and $\ell=0$)
we have
$\ol u(s) = \si_{k_0}(s-a)\le u(s)$ on $(0,b^*)$.  As $u(b^*)=0$, this
implies $\si_{k_0}(b^*-a)\le 0$ so that $\si_{k_0}(s-a)=0$ would have
a solution in $(a,b^*]$, which is not the case.

For (b) the same idea works, use \ref{thm:comp-y'>0} 
to compare $u$ to $\ol u := \co_{k_0}$.
\end{proof}

\begin{thm}\label{thm:3pos_conds}
Let $I$ be an interval and $\ac \cn I \to \R$ continuous. 
Then the Lagrange kernel for $y\mapsto y''' + \ac y'$ is
forward positive in the following cases:
\begin{enumerate}[(a)]
\item $\ac$ is constant, 
\item $\ac \le 0$, or
\item $\ac \le k_1$ for a positive constant $k_1$ with $k_1\le (\pi/L)^2$ where
	$L$ is the length of $I$.
\end{enumerate}
\end{thm}

\begin{proof}
The case of $\ac$  being constant follows form Proposition \ref{prop:cont-kernels}.

The hypothesis of case (b) can be restated as $\ac \le k_1:=0$.
With this notation we have $\ac \le k_1$ in both cases (b) and (c).
Let let $y(s):= K(s;r)$, then
$(y')''+\ac (y')=0$, $y'(a)=0$, $(y')'(a)=1$. 
Applying Lemma \ref{lem:Sturm} to $u:=y'$ yields
$y'>0$ on $I \cap (r,\infty)$ in both cases (b) and
(c).  
Let $\ol K$ be the Lagrange kernel for the $\ol y\mapsto \ol y'''
+ k_1 \ol y'$.  Then $\ol K$ is forward positive by case (a).  
Theorem \ref{thm:comp-y'>0} (with $n=3$ and $\ell=1$) implies $y\ge \ol y$ 
on $I \cap (r,\infty)$.  When $k_1=0$ we have $\ol y(s) = (s-r)^2/2>0$
and when $k_1>0$ we have $\ol y(s) = (1-\co_{k_1}(s-r))/k_1\ge 0$. 
As $y(s)=K(s;r)$ and $r$ was any element of $I$ we are done.
\end{proof}

\begin{lemma}\label{lem:xbar,ybar}
Let $k\in \R$, then the solutions to the initial value problems
\begin{align*}
\ol x_k''' + k \ol x_k'=0,\qquad \ol x_k(0) = 0, \quad\ol x_k'(0)=1, \quad \ol x_k''(0)=0\\
\ol y_k''' + k \ol y_k'=0,\qquad \ol y_k(0) = 0, \quad\ol y_k'(0)=0, \quad \ol y_k''(0)=1
\end{align*}
are
\begin{align*}
\ol x_k(s) &= \si_k(s)\\
\ol y_k(s) &= \begin{cases}
	\dfrac{1-\co_k(s)}{k} ,&k\ne0\\[10pt]
	\dfrac{s^2}{2} ,&k=0.
\end{cases}
\end{align*}
The curve $\ol c_k(s):= (\ol x_k(s), \ol y_k(s))$ has unit affine speed and
constant affine curvature $k$.  This parameterizes the connected component
containing $(0,0)$ of the conic with equation
$$
x^2 + k y^2 - 2y =0.
$$
\end{lemma}

\begin{proof}
Straightforward calculations using equations \eqref{trig_ders} and \eqref{trig_py}.
\end{proof}

\begin{prop}\label{prop:x_bd}
	Let $\ac$ be continuous on $[a,b]$ and let $x\in C^3([a,b])$
	satisfy 
	$$
	x''' + \ac x'=0,\qquad x(a)=0, \quad x'(a) =1, \quad x''(a)=0.
	$$
	Assume for some constants $k_0$ and $k_1$ that $k_0\le \ac \le k_1$ 
	and $k_1\le (\pi/(2(b-a)))^2$.  Then, with the notation of Lemma \ref{lem:xbar,ybar},
	the inequalities 
	$$
	\ol x_{k_1}(b-a) \le x(b) \le \ol x_{k_0}(b-a)
	$$
	hold.
	If equality holds in the lower bound (respectively in the upper bound), 
	then $\ac(s) = k_1$ and $x(s) = \ol x_{k_1}(s-a)$ 
	(respectively $\ac = k_0$ and $x(s) = \ol x_{k_0}(s-a)$) on
	$[a,b]$.
\end{prop}

\begin{prop}\label{prop:y_bd}
	Let $\ac$ be continuous on $[a,b]$ and let $y\in C^3([a,b])$
	satisfy 
	$$
	y''' + \ac y'=0,\qquad y(a)=0, \quad y'(a) =0, \quad y''(a)=1.
	$$
	Assume for some constants $k_0$ and $k_1$ that $k_0\le \ac \le k_1$ 
	and $k_1\le (\pi/(b-a))^2$.  Then, with the notation of Lemma \ref{lem:xbar,ybar},
	the inequalities 
	$$
	\ol y_{k_1}(b-a) \le y(b) \le \ol y_{k_0}(b-a)
	$$
	hold.
	If equality holds in the lower bound (respectively in the upper bound), 
	then $\ac(s) = k_1$ and $y(s) = \ol y_{k_1}(s-a)$ 
	(respectively $\ac = k_0$ and $y(s) = \ol y_{k_0}(s-a)$) on
	$[a,b]$.
\end{prop}

\begin{proof}[Proof of Propositions \ref{prop:x_bd} and \ref{prop:y_bd}] 
	We prove Proposition \ref{prop:x_bd}, the proof of Proposition
	\ref{prop:y_bd} being similar.  By Part (b) of Lemma \ref{lem:Sturm}
	applied to the function $u=x'$ we see that $x'>0$ on $(a,b)$.  
	Therefore the lower bound on $x$ follows form 
    Theorem \ref{thm:comp-y'>0}
	by comparing $x$ to $\ol x(s) = \ol x_{k_1}(s)$ and the upper
	bound by comparing to $\ol x(s) := \ol x_{k_0}(s)$.  Theorem
    \ref{thm:comp-y'>0} also covers the cases of equality.
\end{proof}

%%%%%%%%%%%%%%%%%%%%%%%%%%%%%%%%%%%%%%%%%%%%%%%%%%%%%%%%%%%%%%%%%%%%%%
\section{Comparisons for areas}
\label{sec:Area_comp}
%%%%%%%%%%%%%%%%%%%%%%%%%%%%%%%%%%%%%%%%%%%%%%%%%%%%%%%%%%%%%%%%%%%%%%

\begin{thm}\label{thm:sec-comp}
	Let $D$ be a convex open set in $\R^2$ (which need not be bounded) with
	$C^4$ boundary.  Let $p_0\in \f D$ and let $c\cn [0,L] \to \f D$
	have affine unit speed and $c(0)=p_0$ and let $\ac$ be the affine curvature
	of~$c$.  Let $\ol\ac\cn [0,L]\to \R$ be a continuous function
	so that the operator $\ol y \mapsto \ol y''' + \ol\ac\, \ol y'$
	has forward positive Lagrange kernel 
	and define $\ol A\cn [0,L]\to \R$ by the initial value problem
	$$
	\ol A''' + \ol \ac\ol A' = \frac12,\qquad \ol A(0)=\ol A'(0)=\ol A''(0)=0
	$$
	Then, with $A_{c,0,c(0)}$ as in \eqref{A=int_f},
	\begin{align*}
	\ac(s) \le \ol\ac(s) \text{ on } [0,L] \quad &
	\text{implies}\quad  A_{c,0,c(0)}(L)\ge \ol A(L) \\
	\ac(s) \ge \ol\ac(s) \text{ on } [0,L] \quad &
	\text{implies}\quad  A_{c,0,c(0)}(L)\le \ol A(L) 
	\end{align*}
	In either of these cases if equality holds, then $\ac\equiv \ol\ac$
	on $[0,L]$.
\end{thm}

\begin{proof}
	To simplify notation let $A=A_{c,0,c(0)}$.  
	The by Theorem \ref{thm:A_ode} $A$ satisfies the differential 
	equation $A''' + \ac A = 1/2$ and $p_0=c(0)$ the formulas
	\eqref{A=int_f} and \eqref{A'} imply $A(0)=A'(0)=A''(0)=0$.
	From Equation \ref{A'}, $A'(s) = \frac12(c(s)-c(0))\wedge c'(s)$.
	This implies, see Figure \ref{fig:pos}, $A'(s)> 0$ all $s$ with
	$c(s)\ne p_0$.  Therefore the result follows from Theorem \ref{thm:comp-y'>0}.
\begin{figure}[hb]
\begin{overpic}[width=2.5in]{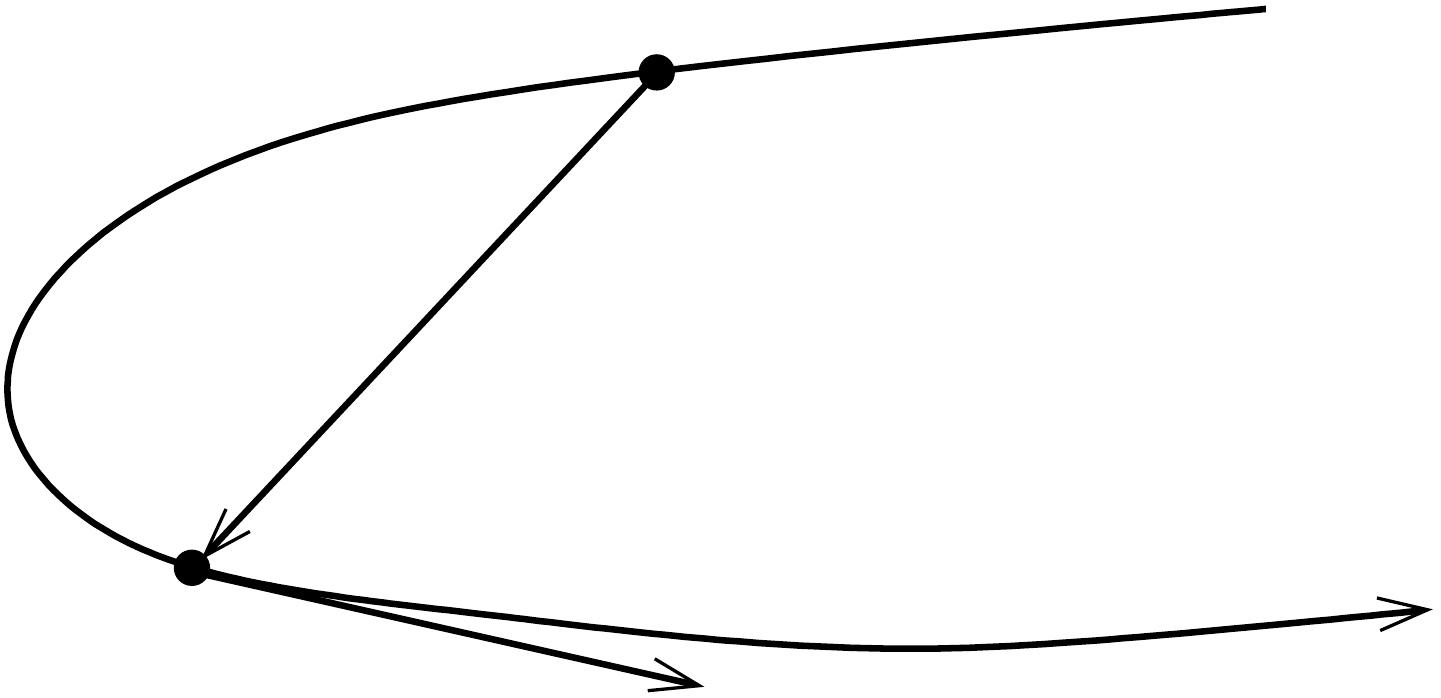}
\put(42,35){$c(0)=P$}
\put(8,3){$c(s)$}
\end{overpic}
\caption{The vectors $c(s)-c(0)$ and $c'(s)$ form a right handed (i.e.\ positive) basis
of $\R^2$.} 
\label{fig:pos}
\end{figure}
\end{proof}

Let $k\in \R$ and set
\begin{equation}  \label{eq:A_bar-def}
\ol A_k(s) = \begin{cases}
\dfrac{s-\si_k(s)}{2k} ,& k\ne0;\\
\dfrac{s^3}{12} ,& k=0.
\end{cases}
\end{equation}
Using the definition of $\si_k$ it is not hard to check
$$
\ol A_k'''(s) + k \ol A_k'(s) =\frac12\qquad \ol A_k(0)= \ol A_k'(0)=\ol A_k''(0) =0.
$$
With the notation of Theorem \ref{thm:sec-comp}, $A_k = A_{c,0,c(0)}$
where $c$ is a curve with constant affine curvature $k$.

\begin{cor}\label{cor:k0-k1-area-bds}
	With  notation as in  Theorem \ref{thm:sec-comp} and Equation
    \ref{eq:A_bar-def} 
        if $k_0\le \ac \le k_1$
	for some constants $k_0$ and $k_1$, then
 \begin{equation*}
	 \ol A_{k_1}(L)\le A_{c,0,c(0)}(s) \le \ol A_{k_0}(L)
 \end{equation*}
	on the interval $[0,L]$. If $A_{c,0,c(0)}(L)= \ol A_{k_0}(L)$ (respectively
	$A_{c,0,c(0)}(L)=\ol A_{k_1}(L)$) then $c$ has constant curvature
	$k_0$ (respectively $k_1$) on the interval $[0,L]$.\qed
\end{cor}

\begin{prop}\label{prop:inside_tri}
Let $\C$ be convex and have an affine curvature bound $\ac \ge k_0$.  
Then any triangle $\triangle p_1p_2p_3$ with vertices on $\C$ satisfies
\begin{equation}\label{eq:tri<Abar}
\area(\triangle p_1p_2p_3) < \ol A_{k_0}(\aff(\C))
\end{equation}
where $\ol A_{k_0}$ as in  Equation \eqref{eq:A_bar-def} and $\aff(\C)$ is
the affine length of $\C$.
\end{prop}

\begin{proof}
Let $k_1$ be any upper bound for the affine curvature of $c$.
Then the result follows from Corollary \ref{cor:k0-k1-area-bds}.
See Figure \ref{fig:triangle1}
\end{proof}

\begin{figure}[hb]
\begin{overpic}[width=2in]{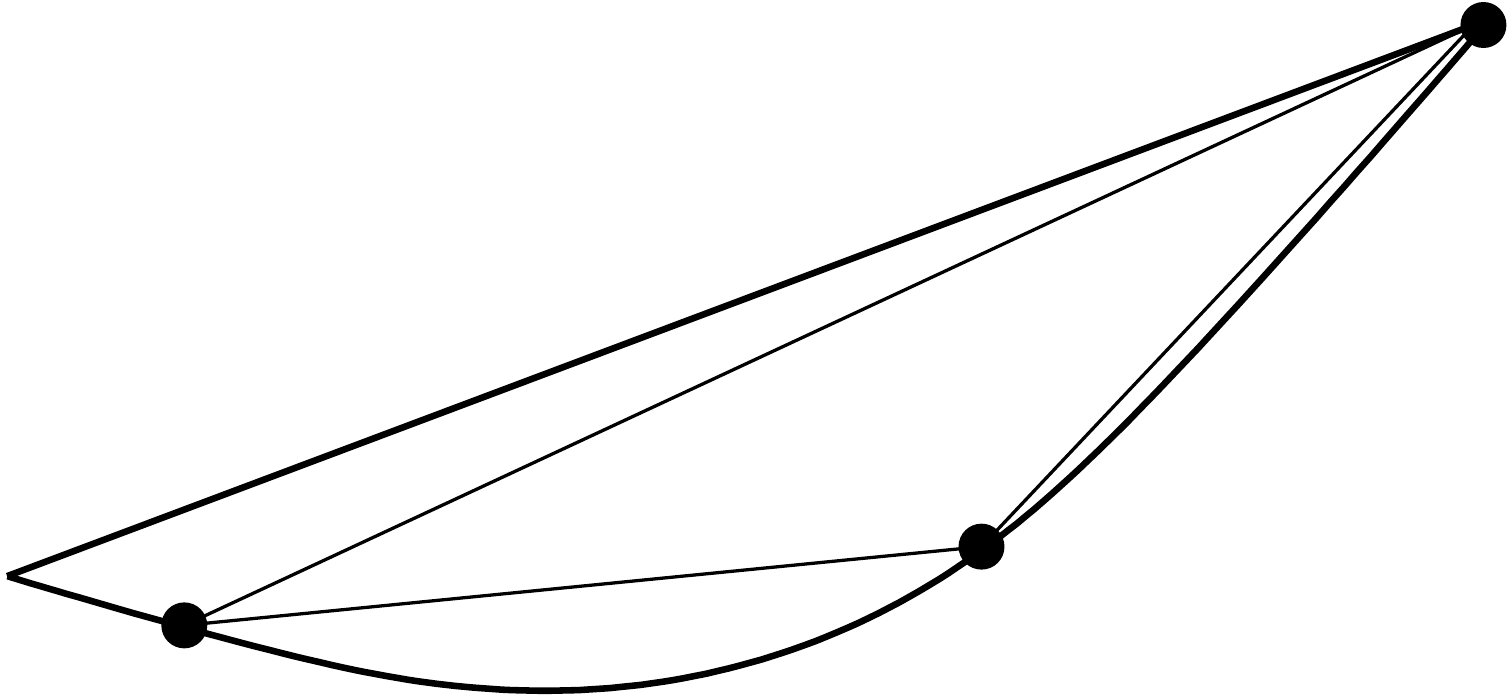}
\put(10,-2){$p_1$}
\put(63,4){$p_2$}
\put(96,36){$p_3$}
% \put(-9,5){$c(a)$}
% \put(101,40){$c(b)$}
\end{overpic}
\caption{The area of $\triangle p_1p_2p_3$ is less than the area
of the region bounded by $\C$ and the segment
between the endpoints of $\C$ and this region has area at most
$\ol A_{k_0}(\aff(\C))$ by Corollary~\ref{cor:k0-k1-area-bds}.}
\label{fig:triangle1}
\end{figure}

\begin{remark}
	The estimate in Proposition \ref{prop:inside_tri} 
	is close to sharp in the sense that for a fixed 
	affine length $\aff(\C)$ as $k_0\to -\infty$
    the ratio of the two sides of inequality \ref{eq:tri<Abar}
	goes to $1$.  To see this let $L>0$ and $\C$ 
	be the curve parameterized by 
	$\ol c_{k_0}\cn [-L,L]\to \R^2$ be the curve of
	Lemma \ref{lem:xbar,ybar}.  Then, as $\ol c_{k_0}$
	has affine unit speed, $\aff(\C)=2L$.  Let
	$p_1$, $p_2$, and $p_3$ be the as in Figure~\ref{fig:low_example}.
	Then a calculation using that $\si_{k_0}(2L) = 2\co_{k_0}(L)\si_{k_0}(L)$
	(which follows from the addition formula \eqref{add_si})
\begin{align*}
	\frac{\area(\triangle p_1p_2p_2)}{\ol A_{k_0}(\aff(\C))}-1 &=
	 - \, \frac{\si_{k_0}(L)-L}{\si_{k_0}(L)\co_{k_0}(L)-L} \\
	&=  -\, \frac{\sinh(\sqrt{|k_0|} L)-\sqrt{|k_0|}L}{\sinh(\sqrt{|k_0|} L)\cosh(\sqrt{|k_0|} L)- \sqrt{|k_0|}L} \\
	&\sim  \frac{- e^{-\sqrt{|k_0|}L}}{2}  =\frac{- e^{-\sqrt{|k_0|}\aff(\C)/2}}{2} .
\end{align*}
\end{remark}

\begin{figure}[ht]
\begin{overpic}[width=2in]{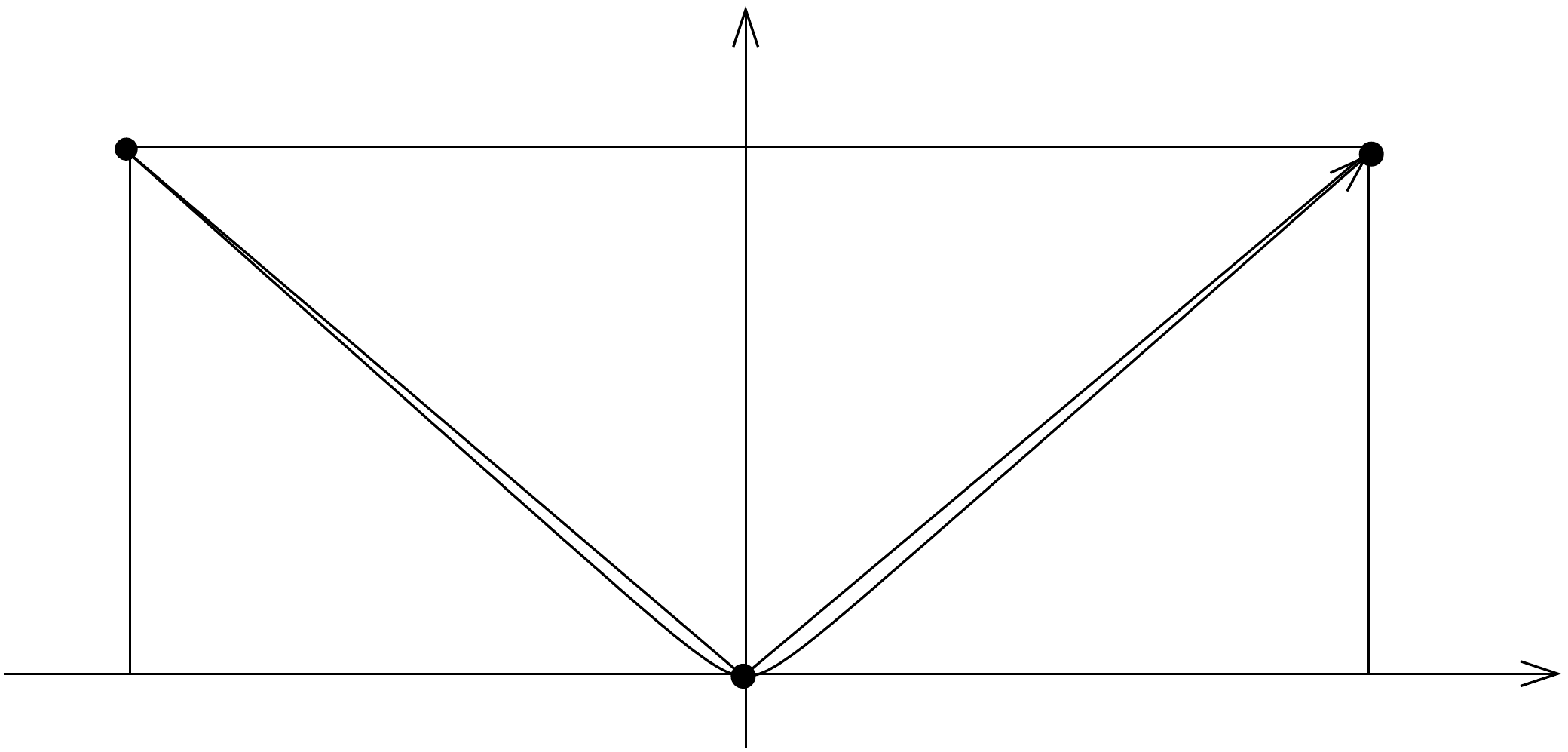}
\put(0,42){$p_1$}
\put(84,42){$p_3$}
\put(49,0){$p_2$}
% \put(-9,5){$c(a)$}
% \put(101,40){$c(b)$}
\end{overpic}
\caption{Here $p_1=(\ol x_{k_0}(-L), \ol y_{k_0}(-L))$, $p_2=(0,0)$,
and $p_3 = (\ol x_{k_0}(L), \ol y_{k_0}(L))$.   Then
$\area(\triangle p_1p_2p_3)= \ol x_{k_0}(L) \ol y_{k_0}(L)$.  
The curve is the hyperbola with equation $x^2 + k_0y^2-2y=0$. }
\label{fig:low_example}
\end{figure}

%%%%%%%%%%%%%%%%%%%%%%%%%%%%%%%%%%%%%%%%%%%%%%%%%%%%%%%%%%%%%%%%%%%%%%
\section{Adapted coordinates and geometric bounds.}
\label{sec:geo_bds}
%%%%%%%%%%%%%%%%%%%%%%%%%%%%%%%%%%%%%%%%%%%%%%%%%%%%%%%%%%%%%%%%%%%%%%

\begin{defn}\label{def:adapted} Let $\C$ be a $C^3$ embedded convex curve and
	$p_0\in \C$.  Let $I_{\C,p_0} \subseteq \R$ be the unique interval so
	that there is an affine unit speed parameterization $c_{\C,p_0}\cn
	I_{\C,p_0}\to \C$ with $c_{\C,p_0}(0)=p_0$ that parameterizes all of
	$\C$.  This is the \bi{standard parameterization} of $\C$ at $p_0$.
	Then the \bi{affine adapted coordinates} to $\C$ at
	$p_0$ are the linear coordinates $\xi$, $\eta$ on $\R^n$ centered at $p_0$
	with
	$$
	\frac{\partial }{\partial \xi}\bigg|_{p_0}=\t_\C(p_0)= c_{\C,p_0}'(0),
	\qquad \frac{\partial }{\partial \eta}\bigg|_{p_0}=\n_\C(p_0)=c_{\C,p_0}''(0).
	$$
	The \bi{graphing parameter set} is 
	the maximal interval $I_{\C,p_0}^* \subseteq I_{\C,p_0}$ 
	so that the image of restriction
	$c_{\C,p_0}\big|_{I_{\C,p_0}^*} \cn I_{\C,p_0}^*\to \C$ is a graph in
	the adapted coordinates $\xi,\eta$.  The \bi{graphing interval}, $I_{\C,p_0}^{**}
	$, is the set $I_{\C,p_0}^{**}:=\{ \xi(c(s)): s\in I_{\C,p_0}^*\}$.
	See Figure \ref{fig:interval}.
\end{defn}

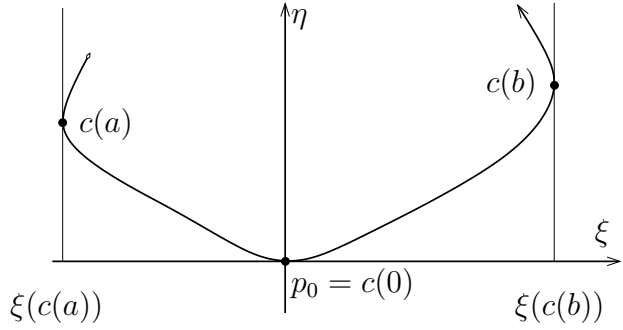
\begin{figure}[h]
\begin{overpic}[width=3in]{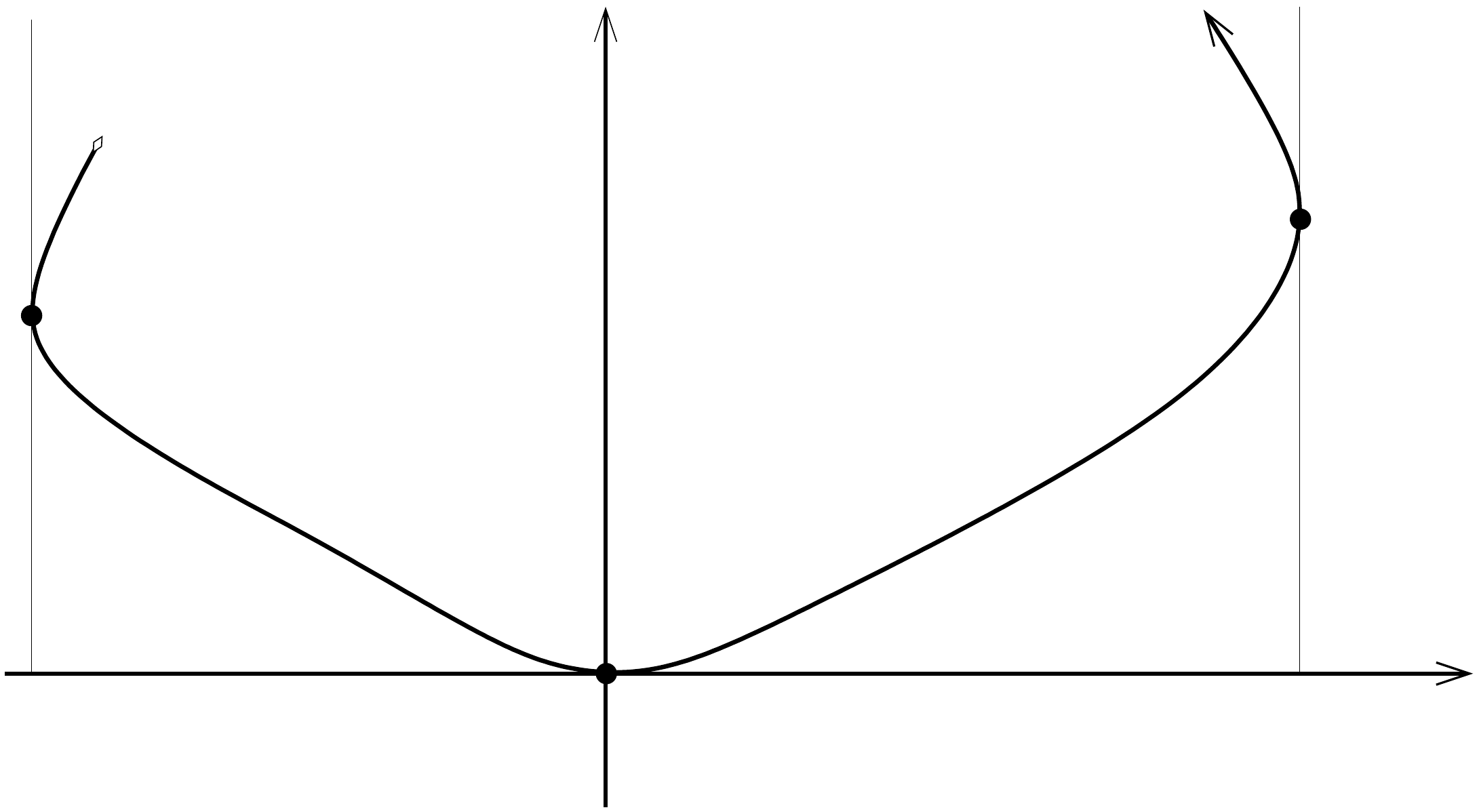}
	\put(5,32){$c(a)$}
	\put(42,4){$p_0=c(0)$}
	\put(76,39){$c(b)$}
	\put(95,13){$\xi$}
	\put(42,51){$\eta$}
	\put(-7,0){$\xi(c(a))$}
	\put(81,0){$\xi(c(b))$}
	\end{overpic}
	\caption{ 
		Here $c:=c_{\C,p_0}$ is the 
		\emph{standard parameterization} of $\C$ at $p_0$.
		The affine adapted coordinates centered at $p_0=c(0)$ 
	have the $\xi$-axis  tangent to $\C$ at $p_0$
	 and the $\eta$-axis is in the direction of the affine normal $c''(0)=\n_\C(0)$. 
	By local convexity the curve is locally the 
	graph of a convex function.  Let $[a,b]$ be the maximal subinterval
	of $I_{\C,p_0}$ with $0\in [a,b]$  so that the restriction $c\big|_{[a,b]}$ is 
	is a graph.  This is $I_{\C,p_0}^*$, the \emph{graphing parameter set}.  The
	interval $[\xi(c(a)),\eta(c(b))]$ on the $\xi$-axis which is the domain
	of the graphing function is the \emph{graphing interval} $I_{\C,p_0}^{**}$.
	Note that the endpoints of $I_{\C,p_0}^{**}$ are (when not an endpoint
	of $I_{\C,p_0}$) are the points where the tangent $\C$ become vertical,
    that is if $x(s)=\xi(c(s))$, the points where $x'(s)=0$.
	} 
	\label{fig:interval}
\end{figure}

\begin{lemma}\label{lem:adapted_xy}
Let $\xi,\eta$ be affine adapted coordinates to $\C$ at $p_0$ and let
$x,y\cn I_{\C,p_0} \to \R$ be the coordinates of $c$ in this coordinate system,
that is $x(s) := \xi(c(s))$ and $y(s):=\eta(c(s))$.  Then $x$ and $y$ satisfy
the initial value problems
\begin{align*}
x'''+\ac x &=0, & x(0)&=0& x'(0)&=1& x''(0)=0\\
y'''+\ac y &=0, & y(0)&=0& y'(0)&=0& y''(0)=1
\end{align*}
where $\ac$ is the affine curvature of $\C$.
\end{lemma}

\begin{proof}
This follows easily  from  the equation $c'''+\ac c'=0$ and
the definition of the adapted coordinates.
\end{proof}

\begin{lemma}\label{lem:graphing}
	Let $p_0\in\C$.  With the notation 
	of Lemma \ref{lem:adapted_xy} the connected component of $0$
	in $\{s\in I_{\C,p_0} : x'(s)>0\}$ is contained in  the graphing
	parameter set, $I_{\C,p_0}^*$.
\end{lemma}

\begin{proof}
	Let $I_0 $ be connected component of $0$ 
	in $\{s \in I_{\C,p_0} : x'(s)>0\}$ and set $J:= \{ x(s): s\in I_0\}$.
	Then, as $x'(s)>0$ for $I_0$, the map $s\mapsto x(s)$ is
	a diffeomorphism between $I_0$ and $J$.  Define $f\cn J\to 
	\R$ by 
    \begin{equation}\label{eq:f_def}	
	f(x(s)):= \int_{0}^s y'(t)\,dt.
\end{equation}	
	Taking the derivative of this gives
	$ \frac{d}{ds} \left( f(x(s))- y(s)\right) =0 $.  Therefore
	$f(x(s))-y(s)=C $ for some constant $C$.  As $x(0)=y(0)=f(0)=0$
	we have $C=0$, which implies $y(s) = f(x(s))$
	for $s\in I_0$.  Therefore $I_0$ is contained in the graphing
	parameter set about $p_0$.
\end{proof}

In the definition of the affine adapted coordinates 
it is assumed the curve $\C$ is of differentiability class
$C^3$.   In general the standard  parameterization will
only be $C^2$.  If the standard parameterization
is $C^3$ which is what is required for the affine curvature
to be defined, there is a gain in the regularity of
$\C$: it will be a $C^4$ immersed submanifold of $\R^2$
as we now show.

\begin{prop}\label{prop:C4}
    With notation as in Lemma \ref{lem:graphing} if
    $\ac$ is $C^k$ for $k\ge0$, then the function $f$
    is $C^{k+4}$.  Thus the curve $\C$ is a $C^{k+4}$
    immersed submanifold of $\R^2$.   
\end{prop}

\begin{proof}
    As $y(s)=f(x(s))$ and $x$ and $y$ are $C^3$ it
    follows $f$ is $C^3$.  
    Taking three derivatives of \eqref{eq:f_def} gives
\begin{align}\nonumber
    y'(s)&= f'(x(s))x'(s)\\
    y''(s)&= f''(x(s))x'(s)^2 + f'(x(s))x''(s)\nonumber\\
    y'''(s) &= f'''(x(s))x'(s)^3 +3 f''(x(s))x'(s)x''(s)
    + f'(x(s))x'''(s).\label{f'''} 
\end{align}
Using $x'''(s) = -\ac(s) x'(s)$, $y'''(s) =- \ac(s) y'(s)$ 
and $y'(s) = f'(x(s))x'(s)$ in \eqref{f'''} gives
$$
    0 = f'''(x(s))x'(s)^3 +3 f''(x(s))x'(s)x''(s)
$$
The function $x'$ does not vanish on the interior of
the graphing parametrization
set so we can divide by $x'(s)^3$ to get
$$
f'''(x(s))= - \frac{3f''(x(s))x''(s)} { x'(s)^2}.
$$
As $x$ is $C^3$, and $x'\ne0$ it has a $C^3$ inverse
$g$, that there is a $C^3$ function $g$ with $g(x(s))=s$.
Therefore
$$
f'''(x) = - \frac{3f''(x)x''(g(x))}{ x'(g(x))^2}
$$
which shows $f'''$ is $C^1$ and therefore $f$ is $C^4$.
For a non-parametric curve of class $C^4$ given as a graph $y=f(x)$ 
the affine curvature is (cf.\ \cite[Page 14, eqn.~(83)] {Blaschke}) 
$$
\ac(x) = -\frac12 \left( \frac1{(f''(x)^\frac23}\right)''= \frac{f''''(x)}{2f''(x)^\frac53  }                       
    - \frac{5f'''(x)^2}{9f''(x)^\frac83}               
$$
which can be rewritten in the form
$$
f''''(x) = 2 \kappa(x)f''(x)^\frac53 + \frac{10}{9} \frac{f'''(x)^2}{f''(x)}
$$
 By standard regularity
theorems for ordinary differential equations if $\ac$ is
$C^k$ then $f$ is $C^{k+4}$.  Or one can just take repeated derivatives
of this equation and use induction to get the result.
\end{proof}

\begin{thm}\label{thm:coord_bds}
Let $\C$ be $C^4$ and let $p_0\in \C$.
Assume $(-L,L) \subseteq I_{\C,p_0}$ for some $L>0$.
Also assume for some constants $k_0$ and $k_1$ 
the affine curvature of $\C$ satisfies the bounds
$k_0\le \ac(s)\le  k_1$ for $-L\le s\le L$ and $k_1\le (\pi/2L)^2$.  
Let $\xi ,\eta$ be affine adapted
coordinates at $p_0$ let $ x(s) = \xi(c(s))$ and $y(s) = \eta(c(s))$.
Then $(-L,L)$ is contained in the graphing parameter set
of $\C$ about $p_0$.  With the notation of Lemma \ref{lem:xbar,ybar} the inequalities
\begin{align}\label{x-bounds}
    \ol x_{k_1}(|s|) & \le |x(s)| \le \ol x_{k_0}(|s|)\\
\ol y_{k_1}(s)& \le y(s) \le \ol y_{k_0}(s)
\label{y-bounds}
\end{align}
hold on the interval $(-L,L)$.
Letting $R := \si_{k_1}(L)$, 
the graphing interval of $\C$ at $s_0$ contains
$(-R,R)$. 
 If $s_1\in (-L,L)$
	with $s_1\ne 0$ and equality holds in either of the lower bounds
	of \eqref{x-bounds} or \eqref{y-bounds} then the restriction
	of $c$ to the interval of points between $0$ and $s_1$ is
	a curve of constant curvature $k_1$.  Likewise if  
	equality holds in either of the upper bounds 
	of \eqref{x-bounds} or \eqref{y-bounds} then the restriction
	of $c$ to the interval of points between $0$ and $s_1$ is
	a curve of constant curvature~$k_0$. 
\end{thm}

\begin{proof}
	Applying Lemma~\ref{lem:Sturm} to the function $u=x'$ on $[0,L)$
	yields $x'>0$ on $[0,L)$.  Applying the same lemma
	to $u(s)=x'(-s)$ on $[0,L)$ (which satisfies $u''(s)+ \ac(-s)u(s)=0$)
	implies $x'\ne 0$ the interval
	$(-L,L)$.  Therefore $x'>0$ on $(-L,L)$, thus
	by Lemma \ref{lem:graphing} the interval $(-L,L)$ is contained
	in the graphing parameter set of $\C$ at $p_0$.   On the interval
	$[0,L)$ the inequalities \ref{x-bounds} and \ref{y-bounds}
	on the interval $[0,L)$ follow from Propositions \ref{prop:x_bd}
	and \ref{prop:y_bd}.  

    To get the inequalities for $s\in (-L,0]$ let $\widetilde x,
    \widetilde y \cn [0,-L)\to \R$ be given by
    $$
        \widetilde x(s) = -x(-s), \qquad \widetilde y(s) = y(-s).
    $$
These satisfy 
\begin{align*}
    \widetilde x'''(s) + \ac(-s) \widetilde x'(s)=0, 
    \qquad \widetilde x(0)=0,
    \quad \widetilde x'(0)=1, \quad \widetilde x''(0)=0,\\
    \widetilde y'''(s) + \ac(-s) \widetilde y'(s)=0, 
    \qquad \widetilde y(0)=0,
    \quad \widetilde y'(0)=0, \quad \widetilde y''(0)=1.
\end{align*}
Again using Propositions \ref{prop:x_bd} and \ref{prop:y_bd} along
with
the definitions of $\widetilde x$ and $\widetilde y$
gives
\begin{align*}
     \ol x_{k_1}(s)&\le -x(-s) \le \ol  x_{k_0}(s)\\
     \ol y_{k_1}(s)&\le y(-s) \le  \ol y_{k_0}(s)
\end{align*}
on $[0,L)$.   Replacing $s$ by $-s$ and using 
$\ol x_{k_j}(-s)=- \ol x_{k_j}(s)$ and $\ol y_{k_j}(-s)= \ol y_{k_j}(s)$
shows the inequalities \eqref{x-bounds} and \eqref{y-bounds}
also hold on $(-L,0]$.  The bound \eqref{x-bounds} implies
the graphing interval contains $(-R,R)$.

The statements about when equality holds follow from the equality
cases in Propositions \ref{prop:x_bd} and \ref{prop:y_bd}.
\end{proof}

\begin{cor}\label{cor:global_graph}
Let $\C$ have the affine curvature bound $\ac\le 0$ and
assume there is an affine unit speed parameterization 
$c\cn \R \to \C$ defined on all of $\R$.  Then $c$ is
bijective and for any point $p_0\in \C$ the curve
is globally a graph $\eta=f(\xi)$ in the affine adapted coordinates at $p_0$.
\end{cor}

\begin{proof}  Let $p_0\in \C$.  Without loss of generality 
	we can assume $c(0)=p_0$.  
	Let $L>0$.  In Theorem \ref{thm:coord_bds} 
    let $k_0 = \min\{ \ac(s): s\in [-L,L]\}$ and $k_1=0$
	to see the graphing parameter set contains 
	$(-L,L)$ and the graphing interval contains
	$(-s_{k_1}(L),s_{k_1}(L)) = (-L,L)$ (as $\si_{k_1}(s) = \si_{0}(s)=s$).
	Letting $L\to \infty$ finishes the proof. 
\end{proof}

\begin{example}
To see that an upper bound on the affine curvature is necessary in
this corollary note for any $k>0$, let $\C$ be the circle with
equation $x^2+y^2 = k^{-3/4}$.  Then $c\cn \R\to \C$ given by
$$
c(s)= \left(k^{-3/4}\cos(k^{1/2}s), k^{-3/4}\sin(k^{1/2}s)\right)
$$
is unit affine speed and $\C$ has constant curvature $k$,
but $c$ is not injective and $\C$ is not globally a graph in 
any coordinate system.
\end{example}

\begin{thm}\label{thm:triangle2}
Let $\C$ have curvature bounds $k_0\le \ac\le k_1$ with
$k_0$, $k_1$ constants with $k_1\le (\pi / \aff(\C))^2$.
Let $p_1,p_2,p_3$ be distant points on $\C$.  Let $L=\aff(\C)/2$.
Then
$$
\area(\triangle p_1p_2p_3)\le \ol x_{k_0}(L)\ol y _{k_0}(L).
$$
Equality holds if and only if $\C$ has constant curvature $k_0$ 
and, after maybe reordering, the points $p_1$, $p_2$, and $p_3$
are the initial point, midpoint, and endpoint of $\C$.
\end{thm}

\begin{proof}
Let $p_0$ be the midpoint of $\C$ and construct the affine adapted coordinates
$\xi$, $\eta$ for $\C$ centered at $p_0$.  Label the points $p_1$, $p_2$
and $p_3$ so that they are increasing order along $\C$.
By
Theorem \ref{thm:coord_bds} the curve lies inside the rectangle
defined by $-\ol x_{k_0}(L)\le \xi\le \ol x_{k_0}(L)$
and $0 \le \eta \le \ol y_{k_0}(L)$ as shown in Figure \ref{fig:triangle2}.
Any triangle in inside this rectangle has area at most half the area
of the rectangle and therefore is at most $\ol x_{k_0}(L)\ol y_{k_0}(L)$.
The only way that equality can hold is if the $p_1$ is the initial 
point of $\C$ and $p_1$ has coordinates $(-\ol x_{k_0}(L), \ol y_{k_0}(L))$,
$p_2=p_0$ is the midpoint of $\C$ and $p_3$ is the endpoint of $\C$
and has coordinates $(-\ol x_{k_0}(L), \ol y_{k_0}(L))$. 
This implies equality holds in the upper bounds of Theorem \ref{thm:coord_bds}
and therefore $\C$ has constant curvature $k_0$.
\end{proof}

\begin{figure}[ht]
\begin{overpic}[width=2.5in]{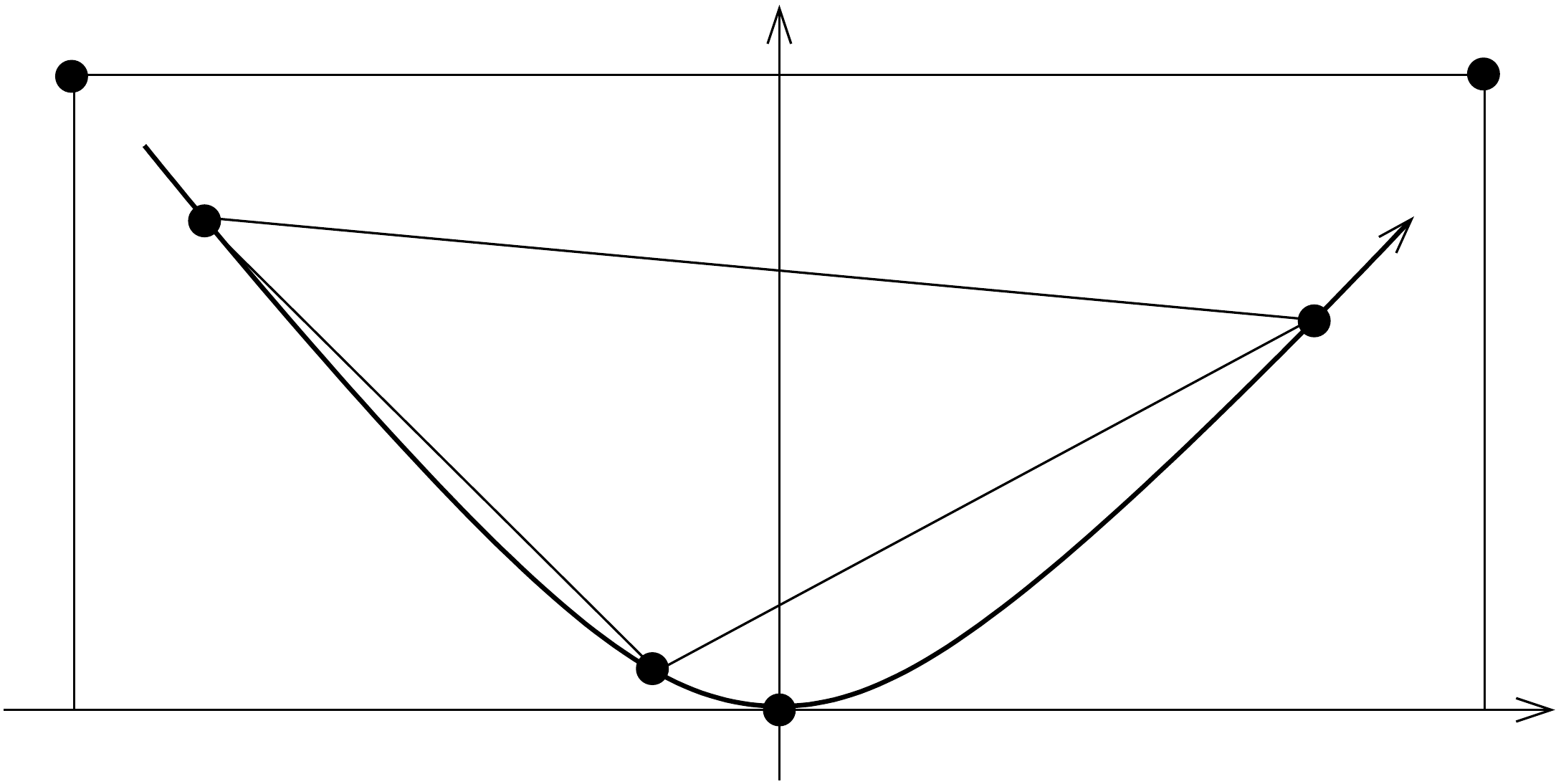}
\put(51,1){$p_0$}
\put(41,11){$p_2$}
\put(13,39){$p_1$}
\put(82,33){$p_3$}
\put(82,33){$p_3$}
\put(98,7){$\xi$}
\put(51,48){$\eta$}
\put(79,48){$(\ol x_{k_0}(L), \ol y_{k_0}(L))$}
\put(-15,48){$(-\ol x_{k_0}(L), \ol y_{k_0}(L))$}
\end{overpic}
\caption{Let $L=\aff(\C)/2$.  Then Theorem \ref{thm:coord_bds}
implies $\C$ is inside the pictured rectangle.  Therefore
$\triangle p_1p_2p_2$ has area at most half the area of this rectangle
and equality only holds when $p_1$ is the upper left corner,
$p_3$ is the upper right corner, and $p_1=p_0$ is on the $\xi$-axis.}
\label{fig:triangle2}
\end{figure}
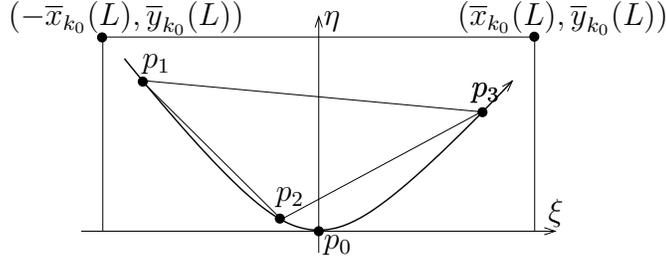

%%%%%%%%%%%%%%%%%%%%%%%%%%%%%%%%%%%%%%%%%%%%%%%%%%%%%%%%%%%%%%%%%%%%%%
\section{Bounds for the number of lattice points on curve.}
\label{sec:lattice}
%%%%%%%%%%%%%%%%%%%%%%%%%%%%%%%%%%%%%%%%%%%%%%%%%%%%%%%%%%%%%%%%%%%%%%

\begin{defn}
Let $v_0, v_1,v_2\in \R^2$ with $v_1$ and $v_2$ linearly independent.
The \bi{lattice with origin $v_0$ and generated by $v_1$ and
$v_2$} is the set
$$
\lat(v_0,v_1,v_2):=\{ v_0+ mv_1+nv_2: m,n\in \Z\}.
$$
\end{defn}
The most basic invariant of a lattice is the area of its fundamental 
domain, $A_\lat$.  If $\lat = \lat(v_0,v_1,v_2)$ this is given by
$$
A_\lat = | v_1\wedge v_2|.
$$

\begin{prop}\label{prop:lat-triangle}
If $p_1,p_2,p_3$ are three nonlinear points in $\lat = \lat(v_0,v_1,v_2)$,
then the area of the triangle $\triangle p_1p_2p_3$ satisfies 

$$
\area(\triangle p_1p_2p_3) = \frac{\m}2 A_\lat.
$$
for some positive integer $\m$ and therefore  $\triangle p_1p_2p_3\ge 
A_\lat/2$.
\end{prop}

\begin{proof}
If $p_j = v_0 + m_jv_1 + n_jv_2$, then 
\begin{align*}
\area(\triangle& p_1p_2p_3) = \frac12|(p_2-p_1)\wedge (p_3-p_1)|\\
&= \frac12|(m_2-m_1)(n_3-n_1) - (n_2-n_1)(m_3-m_1)||v_1\wedge v_2|\\
&= \frac{\m}2 A_\lat
\end{align*}
where $\m=|(m_2-m_1)(n_3-n_1) - (n_2-n_1)(m_3-m_1)|$ is a positive integer.
\end{proof}

As an example to motivate the following definition 
consider the ellipse defined by $ax^2+bxy+cy^2=R$ where
$a,c,R$ are odd integers and $b$ is an even integer.
Let $\lat =\Z^2$ be the standard integral lattice.
If $p=(x,y)\in \lat\cap \C$ then reducing 
$ax^2 +bxy+cy^2 = R^2$ modulo $2$ (and using 
$x^2\equiv x\space\operatorname{mod} 2$) gives 
$x+y\equiv 1 \operatorname{mod}2$.  Whence 
$p\equiv (1,0)$ or $p\equiv (0,1)$ modulo $2$.
Thus if  $p_j = (x_j,y_j)\in \lat\cap \C$ for $j=1,2,3$,
then for at least at least one pair from $\{p_1,p_2,p_3\}$,
say $p_1,p_2$ we have $p_1\equiv p_2 \operatorname{mod} 2$.
This implies $(p_2-p_1)\wedge (p_3-p_1)$ is
even and therefore 
$\area(\triangle p_1p_2p_3)=\frac12 |(p_2-p_1)\wedge (p_3-p_1)|$
is an integer and whence is twice as large as $1/2$ which
the minimum area of a general triangle with vertices
in $\Z^2$.

\begin{defn}\label{def:m_C}
    Let $\lat$ be a lattice and $\C$ a curve.  Then $\m(\C,\lat)$
    is the largest positive integer so that 
    $$
    \area(\triangle p_1p_2p_2) \ge \frac{\m(\C,\lat)}{2} A_\lat
    $$
    for all distinct points $p_1,p_2,p_3 \in \lat\cap \C$.
\end{defn}

Defining and using the integer $\m(\C,\lat)$ 
to improve lattice point estimates is an abstraction
of an idea in the paper \cite{Ramana} of Ramana where a
related integer, $m_{ad}$ is defined for the integer lattice
and integral conics of the form $ax^2 + dy^2 = R$ and 
the lattice $\Z^2$.

\begin{defn}\label{def:F_bar}
For each $k\in \R$ let $\ol F_k$ be the inverse of function $\ol A_k$.  (The
function $\ol A_k$ is strictly increasing and therefore this inverse exists.)
\end{defn}

\begin{thm}\label{thm:2pts1}
Let $\C$ have be convex, $C^4$, and with affine length $\aff(\C)$
and a lower bound $\ac \ge k_0$ on its affine curvature.  Let $\m=\m(\C,\lat)$.
If 
\begin{equation}\label{eq:Ak0L}
\ol A_{k_0}(\aff(\C)) \le \frac{\m A_\lat}{2} 
\end{equation}
then, there are at most two points of $\lat$ on $\C$.
\end{thm}

\begin{proof}
	If there are three or more points of $\lat$ on $c$, let
	$p_1$, $p_2$, and $p_3$ be three of them.  By Propositions
	\ref{prop:lat-triangle} and  \ref{prop:inside_tri} 
    and the definition of $\m=\m(\C,\lat)$ 
	this implies
	$$
    \frac{\m A_\lat}{2} \le \area(\triangle p_1p_2p_3) < \ol{A}_{k_0}(\aff(\C))
	$$
	which contradicts \eqref{eq:Ak0L}.  
\end{proof}

\begin{thm}\label{thm:low_aff_bd}
Let $\lat$ be a lattice, $\C$ a $C^4$ curve
 with a lower curvature bound $\ac \ge k_0$ on
its affine curvature, and $\m=\m(\C,\lat)$ as in definition
\ref{def:m_C}.  Assume
that $\C$ is convex, or more generally that any sub-arc of
affine length at most $F_{k_0}(A_\lat/2)$ is convex. 
Then
$$
\# (\lat \cap \C) \le 2\left\lceil \frac{\Lambda(\C)}{F_{k_0}(\m A_\lat/2)} \right\rceil.
$$
\end{thm}

\begin{proof}
	By Theorem \ref{thm:2pts1} any sub-arc of $\C$ with affine length 
	at most $F_{k_0}(\m A_\lat/2)$ contains at most two points of $\lat$.
	Let $m = \left\lceil \Lambda(\C)/ F_{k_0}(\m A_\lat/2)\right\rceil$. 
    By dividing $\C$ 
	into $m$ sub-arcs of affine length $L/m$  we cover $c$
	with $m$ sub-arcs with affine length $\le F_{k_0}(\m A_\lat/2)$.
	As each of these sub-arcs contains at most two points of 
	of $\lat$ the total number of points in $\lat \cap \C $
	is at most $2m$.
\end{proof}

\begin{thm}\label{thm:2pts2}
Let $\C$ have affine curvature bounds $k_0\le \ac\le k_1$ with $k_1\le (\pi/\Lambda(\C))^2$.
Let $\lat$ be a lattice and let $L =\aff(\C)/2$ and $\m = \m(\C,\lat)$.  If
\begin{equation}\label{eq:xy<=A/2}
\ol x_{k_0}(L)\ol y_{k_0}(L)\le \frac{\m A_\lat}{2},
\end{equation} 
then
$$
\#(\lat \cap \C) \le 3.
$$
The equality  $\#(\lat \cap \C) =3$ holds if and only if 
$\C$ has constant curvature 
$k_0$, equality holds in \eqref{eq:xy<=A/2} 
and the three points of $\lat$ on $\C$ are the two endpoints
of $\C$ along with the midpoint of $\C$.
\end{thm}

\begin{proof}
If there are more than two points of $\lat$ on $\C$, let $p_1$, $p_2$, $p_3$
be three of them.  By Theorem \ref{thm:coord_bds} these points
are on a convex graph and therefore are not collinear.  Thus
by Proposition \ref{prop:lat-triangle}, Theorem \ref{thm:triangle2},
the definition of $\m$ and the inequality \eqref{eq:xy<=A/2} 
$$
\frac{\m A_\lat}{2} \le \area(\triangle p_1p_2p_3) \le \ol x_{k_0}(L)\ol y_{k_0}(L)\le
\frac{\m A_\lat}{2}.
$$
Therefore equality holds in Theorem \ref{thm:triangle2}, which happens if 
and only if these three points are the midpoint of $\C$ along with the
endpoints of $\C$ and $\C$ has constant affine curvature $k_0$.  This
shows that any size three subset of $\lat\cap \C$ consists of
the endpoints and midpoint of $\C$ and thus $\C$ has at most
three points.
\end{proof}

\begin{lemma}\label{lem:x_bar_y_bar_homo}
	Let $k\in \R$ and define intervals $I_k$ and $J_k$ by
$$
I_k := \begin{cases}
	[0,\infty),& k\le 0;\\
	[0,\pi/2\sqrt k\,],& k>0.
	\end{cases}\hspace{.5in}
J_k:= \begin{cases}
	[0,\infty),& k\le 0;\\
	[0,1/k^{3/2}],& k>0.
	\end{cases}
$$
and let $H_k$ be defined on $I_k$ by $H_k(s) = \ol x_k(s)\ol y_k(s)$.
Then $H_k$ is a homeomorphism between $I_k$ and $J_k$.
\end{lemma}

\begin{proof}
	It is elementary to check that each of $\ol x_k$ and $\ol y_k$
	are strictly increasing on $I_k$, and therefore $H_k$ is 
	also strictly increasing.  Also $\lim_{s\to \infty} H_k(s)=\infty$
	when $k\le 0$ and $H_k(\pi/2\sqrt k) = 1/k^{3/2}$ when $k>0$.
	This implies $H_k$ is a bijective continuous map between
	the intervals and thus a homeomorphism.
\end{proof}

\begin{defn}\label{def:G_k}
Let $G_k\cn J_k\to I_k$ be the inverse of the map $H_k$ of Lemma~\ref{lem:x_bar_y_bar_homo}

\end{defn}

\begin{thm}\label{thm:sharp_lat}  
    Let $\C$ be a curve with affine curvature
    bounds $k_0\le \ac \le k_1$.  Let $\lat$ be a lattice, 
    $\m = \m(\C,\lat)$, and set $L =
    G_{k_0}(\m A_\lat/2)$ and $m = \lfloor \Lambda(\C)/(2L)\rfloor$.  
    If $k_1>0$ also assume $k_1\le (\pi/(2L))^2$.  Then 
	$$
	\#( \lat \cap \C) \le 2m +2.
	$$
\end{thm}

\begin{proof}
		and let $c\cn [0,\aff(\C)] \to \R^2$ be an affine unit speed parameterization
	of $\C$.  For $j=1,2 ,\ldots, m$ define sub-arcs of $\C$ by
	\begin{align*}
		\C_j &:= \{c(s): 2(j-1)L \le s <  2jL\}\\
		\C_j^*&:= \{ c(s) : 2(j-1)L \le s \le  2jL\}\\
		\C_{m+1}&:= \{ c(s) : 2mL \le s \le \aff(\C)\}.
	\end{align*}
	Then $\C_j$ is just $\C_j^*$ with its right endpoint removed.
	The affine length of $\C_j^*$ is $2L=G_{k_0}(\m A_\lat/2)$ 
	and by the definition $G_{k_0}$ we have $\ol x_{k_0}(L)\ol y_{k_0}(L)=
	\m A_\lat/2$. Therefore by Theorem \ref{thm:2pts2} the arc $\C_j^*$
	contains at most $3$ points of $\lat$ and if it does contain $3$ points,
	then two of these points are endpoints of $\C_j^*$.  Thus $\C_j$
	contains at most two points of $\lat$.  The arc $\C_{m+1}$ 
	has affine length less than $2L$, using Theorem \ref{thm:2pts2} again,
	it contains at most two points of $\lat$.  As 
	$\C = \C_1 \cup \C_2\cup \cdots \cup \C_{m+1}$ this implies
	$\C\cap \lat$ contains at most $2(m+1)$ points.
\end{proof}

There is a rigidity version of this result.

\begin{thm}\label{thm:rigid_lat}
	Let $\C$ satisfy the hypothesis of Theorem \ref{thm:sharp_lat}
	with the extra assumption that
	$$
	m =  \frac{\aff(\C)}{2L}
	$$
	is an integer.  Then if $\C$ is open (that is not the boundary of a 
    bounded convex
	domain) then
	$$
	\#(\lat \cap \C) \le 2m+1.
	$$
	Equality holds if and only if $\C$ has constant curvature
	$k_0$ and the points of $\lat \cap \C$ are evenly spaced along 
	$\C$ with respect 
	to affine arc length at a distance of $L=G_{k_0}(\m A_\lat/2)$ 
	between consecutive points. In particular the endpoints of $\C$
	are in $\lat$.
\end{thm}

\begin{proof}
	That $m$ is an integer implies in the proof of Theorem \ref{thm:sharp_lat} 
	that $\C_{m+1}$ is just the one point set $\{c(\aff(\C))\}$.
	Thus each $\C_j$ contains at most two points of $\lat$ for
	$j=1,2 ,\ldots, m$ and $\C_{m+1}$ contains at most one point
	of $\lat$.  Thus $\C =  \C_1 \cup \C_2\cup \cdots \cup \C_{m+1}$
	contains at most $2m+1$ points of $\lat$.

	We prove that when equality holds that $\C$ has constant curvature and
	that the points  of $\C\cap \lat$ are evenly spaced by induction on
	$m$.  If $m=1$, then the result follows from Theorem \ref{thm:2pts2}
	because $\ol x_{k_0}(L)\ol y_{k_0}(L)=\m A_\lat/2$.  

	Assume the result holds
	for $m$ and let $\C$ be a curve with 
	$\aff(\C) = (m+1) G_{k_0}(\m A_\lat/2)$ and $\#(\lat \cap \C)=2m+3$.
	Let $c\cn [0, \aff(\C)]\to
	\C$ be an affine unit speed parameterization of $\C$ and
	let $\C':= \{ c(s): 0\le s\le 2L\}$ and 
	$\C'':=\{ c(s): 2L \le s \le (m+1)L=\aff(\C)\}$. 
	Then $\C'$ has at most $3$ points and $\C''$ has at most
	$2m+1$ points.  The intersection $\C'\cap \C''$ only
	has the one point $c(2L)$.  If this point is not in $\lat$,
	then by the induction hypothesis $\#(\lat \cap \C')\le 2$
	and $\#(\lat \cap \C'')\le 2m$ and the set $\lat \cap \C'$
	and $\lat \cap \C''$ have no point in common.  Thus
	$\#(\lat \cap \C) = \#( \lat\cap (\C'\cup \C'')) \le 2 + 2m$
	contradicting that $\#(\lat \cap \C)=2m+3$.  Therefore
	$\lat\cap \C'$ and $\lat\cap \C''$ have one point in common, whence
	$$
	\#(\lat \cap \C) = \#(\lat \cap \C')+\#(\lat \cap \C'')-1
	\le 3 + 2m+1 - 1 = 2m+3.
	$$
	Thus the assumption $\#(\lat \cap \C)=2m+3$ implies
	$\#(\lat \cap \C')=3$ and $\#(\lat \cap \C'')=2m+1$.  Therefore
	the induction hypothesis implies $\C'$ and $\C''$, and therefore
	$\C$, have constant affine curvature $k_0$ and the points
	are equality spaced at a distance of $L$ between consecutive 
	points.
\end{proof}

\section{Examples.}
\label{sec:examples}

In this section we give examples to show our theorems bounding the number of
lattice points on a curve are sharp. To simplify things in all of these
examples 
will have $\m(\lat,\C) = 1$.

\subsection{Examples when $k_0=0$.}
In this case the function $H_{k_0}=H_0$ of Lemma \ref{lem:x_bar_y_bar_homo}
is given by  $H_0(s) = s^3/2$ and therefore its inverse (cf.\ Definition
\ref{def:G_k}) is 
$$
G_0(s)= (2s)^{1/3}.
$$
Let $\lat = \lat(v_0,v_1,v_2)$ be a lattice.    By
possibly replacing $v_2$ by $-v_2$ we may assume $v_1\wedge v_2>0$.
Let $\alpha = (v_1\wedge v_2)^{-1/3}$ and let $\C$ be the parabola
parameterized by 
$$
c(s) = v_0 + (\alpha s) v_2 + \frac{(\alpha s)(\alpha s+1)}{2} v_2.
$$
A bit of calculation shows
$$
c'(s)\wedge c''(s) = \alpha^3 v_1\wedge v_2 =1.
$$
Therefore $c$ is an affine unit speed parameterization of $\C$.
Let $s_j=j/\alpha$.  Then
$$
p_j:= c(s_j) = v_0 + j v_1 + \frac{j(j+1)}{2} v_2 \in \lat
$$
and the affine distance between $p_{j+1}$ and $p_j$ is
$s_{j-1}-s_j=1/\alpha$.
The $L$ of Theorems \ref{thm:sharp_lat} and \ref{thm:rigid_lat}  is
given by

$$
L=G_0(A_\lat/2) = (A_\lat)^{1/3} = (v_1\wedge v_2)^{1/3} = 1/\alpha, 
$$ 
which
is the affine distance between $p_j$ and $p_{j+1}$ on $\C$. 

Letting $k_0=0$ and $k_1>0$ with $k_1\le (\pi/2L)^2$ we then have that the
restriction $\C\big|_{p_1}^{p_{2m+1}}$, that is the arc of $\C$ between $p_1$
and $p_{2m+1}$, gives an example where equality holds in Theorem
\ref{thm:rigid_lat} and thus also in Theorem \ref{thm:2pts2}.
The curve $\C\big|_{p_0}^{p_{m+1}}$ is an example where equality holds
in Theorem \ref{thm:sharp_lat}.

%%%%%%%%%%%%%%%%%%%%%%%%%%%%%%%%%%%%%%%%%%%%%%%%%%%%%%%%%%%%%%%%%%%%%%
\subsection{Constructing closely spaced lattice points on conics.}
%%%%%%%%%%%%%%%%%%%%%%%%%%%%%%%%%%%%%%%%%%%%%%%%%%%%%%%%%%%%%%%%%%%%%%

\begin{lemma}\label{lem:lat=lat'}
    Let $\lat = \lat(v_0,v_1,v_2)$ and $\lat'=\lat(v_0',v_1',v_2')$
    Assume $\lat' \subseteq \lat$ and $A_\lat = A_{\lat'}$.
    Then $\lat = \lat'$.
\end{lemma}

\begin{proof}
    As  $\lat' \subseteq \lat$ we have $v_0'\in \lat$ and
    therefore  we can write $\lat$ as $\lat=\lat(v_0',v_1,v_2)$.
    Then $\lat' \subseteq \lat$ implies there are integers $a_{ij}$ so 
    that
    \begin{align*}
        v_0'+v_1'&= v_0' + a_{11}v_1+a_{12}v_2\\
        v_0'+v_2'&= v_0' + a_{21} v_1 + a_{22} v_2
    \end{align*}
    The equality
    $A_\lat = A_{\lat'}$ implies $v_1\wedge v_2= \pm v_1'\wedge v_2'$.
    Therefore 
    \begin{align*}
        v_1'\wedge v_2'&= ( a_{11}v_1+a_{12}v_2)\wedge(a_{21} v_1 + a_{22} v_2)\\
                       &= (a_{11}a_{22}- a_{12}a_{21})v_1\wedge v_2 \\
                       &= \pm (a_{11}a_{22}- a_{12}a_{21})v_1'\wedge v_2'
    \end{align*} 
    Thus $a_{11}a_{22}- a_{12}a_{21}= \pm1$.  Whence if $[b_{ij}]= [a_{ij}]^{-1}$
    is the inverse of the matrix $[a_{ij}]$, then, using Cramers's rule
    for the inverse, we see the numbers $b_{ij}$ are integers.
    Then
     \begin{align*}
        v_0'+v_1&= v_0' + b_{11}v_1'+b_{12}v_2'\\
        v_0'+v_2&= v_0' + b_{21} v_1 '+ b_{22} v_2'
    \end{align*}
    This implies $\lat \subseteq \lat'$ and thus $\lat = \lat'$.
\end{proof}

\begin{lemma}\label{lem:prservess}
  Let $\lat$ be a lattice and $\phi$ an affine motion of $\R^2$ so
  that for some points $p_0, p_1, p_2\in \lat$ we have
  $\phi(p_0),\phi(p_1), \phi(p_2)\in \lat$ and 
 $$
     \area(\triangle \phi(p_0)\phi(p_1)\phi(p_2)) = \frac{A_\lat}{2.}
 $$ 
 Then $\phi$ preserves the lattice $\lat$.
\end{lemma}

\begin{proof}
As $\phi$ is an affine motion it preserves area.  Therefore
$$
\area(\triangle p_0p_1p_2) = \area(\triangle \phi(p_1)\phi(p_2)\phi(p_3))=
A_\lat/2.
$$
As 
$p_0,p_1,p_2\in \lat$ we have
$\lat_1:=\lat(p_0, p_1-p_0, p_2-p_0) \subseteq \lat$.  Then
$\area(\triangle p_0p_1p_2) = A_\lat/2$ implies
$A_{\lat_1} = A_\lat$ and by Lemma \ref{lem:lat=lat'} 
we have $\lat_1=\lat$.  The same argument shows $\lat
= \lat(\phi(p_0), \phi(p_1)-\phi(p_0), \phi(p_2)-\phi(p_0)) = \lat$.
Thus the image of $\lat$ under $\phi$ is
\begin{align*}
    \phi[\lat]&= \phi[ \lat(p_0,p_1-p_0, p_2-p_0)]\\
              &= \lat(\phi(p_0),\phi(p_1)-\phi(p_0), \phi(p_2)-\phi(p_0))\\
          &=\lat
\end{align*}
as required.
\end{proof}

\begin{prop}\label{prop:equal_spaced}
    Let $\C$ be connected component of a conic  
    with constant affine curvature $k_0$ and
    let $\lat$ be a lattice.  Assume there are distinct points let
    $p_1, p_2, p_3, p_4$ listed in increasing order with
    respect to the natural orientation on $\C$ with the affine distance between $p_j$
    and $p_{j+1}=L$ for some constant $L$ and $j=1,2,3$ and with
    $$
        \area(p_2,p_3,p_4) = \frac{A_\lat}{2}.
    $$  
    Then there is a unique special affine motion $\phi$ which preserves
    both the curve $\C$ and the lattice $\lat$ and with
    $\phi(p_j)= p_{j+1}$ for $j=1,2,3$.  Also 
    \begin{enumerate}[\  (a)]
        \item for all integers $j$ we have $\phi^j(p_1)\in \lat\cap \C$ and
            for all $j$ the affine distance between $p_j:=\phi^{j-1}(p_1)$
            and $p_{j+1}:= \phi^j(p_1)$ is $L$.
        \item if $k_0$ is the curvature of $\C$ then $k_0$, $L$ 
            and $A_\lat$ are related by
            \begin{equation}\label{eq:A_L=H_k} 
                \frac{A_{\lat}}{2}= H_{k_0}(L).
    \end{equation} 
        where $H_{k_0}$ is as in Lemma \ref{lem:x_bar_y_bar_homo}.
    \end{enumerate}
\end{prop}

\begin{proof}
   Let $c_1\cn \R\to \C$ be an affine unit speed parametrization of
   of $\C$ with $c_1(0)=p_1$. Then $p_j=c((j-1)L)$ for $j=1,2,3,4$.
   Let $c_2(s)=c_1(s+L)$.  Then $c_2$ is an affine unit speed
   parameterization of $\C$.  As $\C$ has constant affine curvature
   Theorem \ref{thm:unq} gives us a unique affine motion $\phi$ 
   with $c_2(s) = \phi(c_1(s))$.  Then $\phi(c_1(s))= c_1(s+L)$
   for all $s$ and thus $\phi$ preserves $\C$ and 
   $\phi(p_j)=p_{j+1}$ for $j=1,2,3$.  Also 
   $$
   \area(\triangle \phi(p_1) \phi(p_2)\phi(p_3)) = \area(\triangle p_2p_3p_4)
   = \frac{A_\lat}{2}.
   $$
   Therefore, by Lemma \ref{lem:prservess}, $\phi$ preserves $\lat$.
   For all integers $j$ we have $\phi^j(p_1) = c_1(jL)\in \lat$ and
   $\phi$ preserves affine distance, therefore the affine 
   distance between $p_{j}$
   and $p_{j+1}$ is $L$.
   Finally equation \eqref{eq:A_L=H_k} holds by Theorem~\ref{thm:2pts2}.
\end{proof}

\subsection{Examples when $k_0<0$.}
%%%%%%%%%%%%%%%%%%%%%%%%%%%%%%%%%%%%%%%%%%%%%%%%%%%%%%%%%%%%%%%%%%%%%%

We first consider the case where the lattice is $\lat = \Z\times \Z$
is the lattice of integers points in the plane.  Let $\C_0$ be
the connected component of the hyperbola with equation
$$
    x^2-xy-y^2=1
$$
which contains the point $(1,0)$.  Let 
$$
\alpha := 2^{-1/3}5^{1/6}
$$
Then a calculation shows that $c \cn \R\to \R^2$ given by
$$
    c(s) = \left( \cosh(\alpha s) - \frac{1}{\sqrt5}\sinh(\alpha s), - \frac{2}{\sqrt5}\sinh(\alpha s) \right)
$$
is a affine unit speed parameterization of $\C_0$ with  
$$
    c(0) = (1,0).
$$
It is not hard to see that
$$
    c'''(s) = \alpha^2 c'(s)
$$
and therefore $\C$ has constant curvature $k_0 = -\alpha^2$.  
Let 
$$
L = \frac{1}{\alpha} \operatorname{arcsinh}(\sqrt5 / 2).
$$
Then $\sinh(\alpha L) = \sqrt5/2$ and 
$\cosh (\alpha L) = \sqrt{(1 + (\sqrt5 /2)^2)}= 3/2$.  Therefore
$$
    c(L) = (1,-1).
$$
To find more integral points on $\C$ let $\phi$ be the 
linear  map $\phi \cn \R^2\to \R^2$ 
given by
$$
    \phi(x,y) := (x-y,-x+2y).
$$
Then $\phi$ is an affine motion and preserves both $\C$ and $\lat$.  
Define $p_1 :=c(0)=(1,0)$ and $p_2 := c(L) = (1,-1)$, 
then $\phi(p_1) = p_2$.  It follows that for all $j\in \Z$
the 
points\footnote{For $j>0$ it is not hard to check if $f_0,f_1, f_2\ldots $
    is the Fibonacci sequence defined by $f_{j+2} = f_{j+1} + f_j$, $f_0=0$,
    $f_1=1$, then for $j\ge 2$ the points
are $p_j=(f_{2j-3}, - f_{2j-2})$}.
$p_j:= c((j-1)L) = \phi^{j-1}(p_1)$ are
in $\lat \cap \C$.   Then $p_2 = (2,-3)$ and
$p_4= (5,-8)$, then $\area(\triangle p_2p_3p_4)=1/2= A_\lat/2$
(See Figure \ref{fig:hyp_example}).  
Therefore Proposition \ref{prop:equal_spaced} implies 
$H_{k_0}(L) = A_\lat/2$, or what is the same thing $G_{k_0}(A_\lat/2) = L$.

\begin{example}\label{ex:ZxZ}
    If $\C_1:= \{ c(s): 0 \le s \le (2m+1)L\}$, then this gives an example where
    equality holds in Theorem \ref{thm:sharp_lat}.   Letting $\C_2:=\{
    c(s): L \le s \le (2m+1)L\}$ gives an example where equality holds in
    Theorem \ref{thm:rigid_lat}.
\end{example}

\begin{example}
The previous example can be transferred to other lattices.
Let $\lat = \lat(v_0,v_1,v_2)$.  We can assume $v_1\wedge v_2>0$
(if not replace $v_2$ by $-v_2$). 
With $\alpha$ and $L$ as in the previous example let
$$
\beta = (v_1\wedge v_2)^{1/3} 
$$
and
$$
\widehat c(s) =  
v_0 + \left( \cosh(\alpha\beta s) - \frac1{\sqrt 5} \sinh(\alpha\beta s)  \right) v_1 - \left(  \frac{2}{\sqrt 5} \sinh(\alpha\beta s)  \right) v_2.
$$
Let $\widehat L = L/\beta$.  Then the curve 
$\{ \widehat c(s): \widehat L \le s\le  
(2m+1)\widehat L\}$ is an example 
where equality holds in Theorem \ref{thm:rigid_lat}
and for the curve $\{ \widehat c(s): 0 \le s \le
(2m+1)  \widehat L\}$ equality holds in Theorem \ref{thm:sharp_lat}.
\end{example}

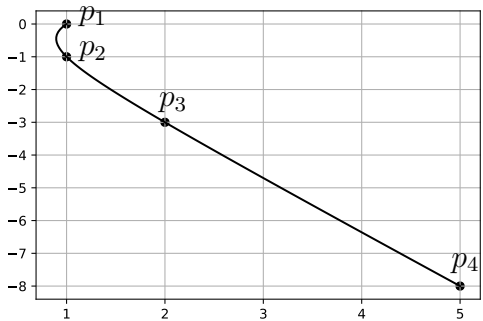
\begin{figure}[h]
\begin{overpic}[width=3in]{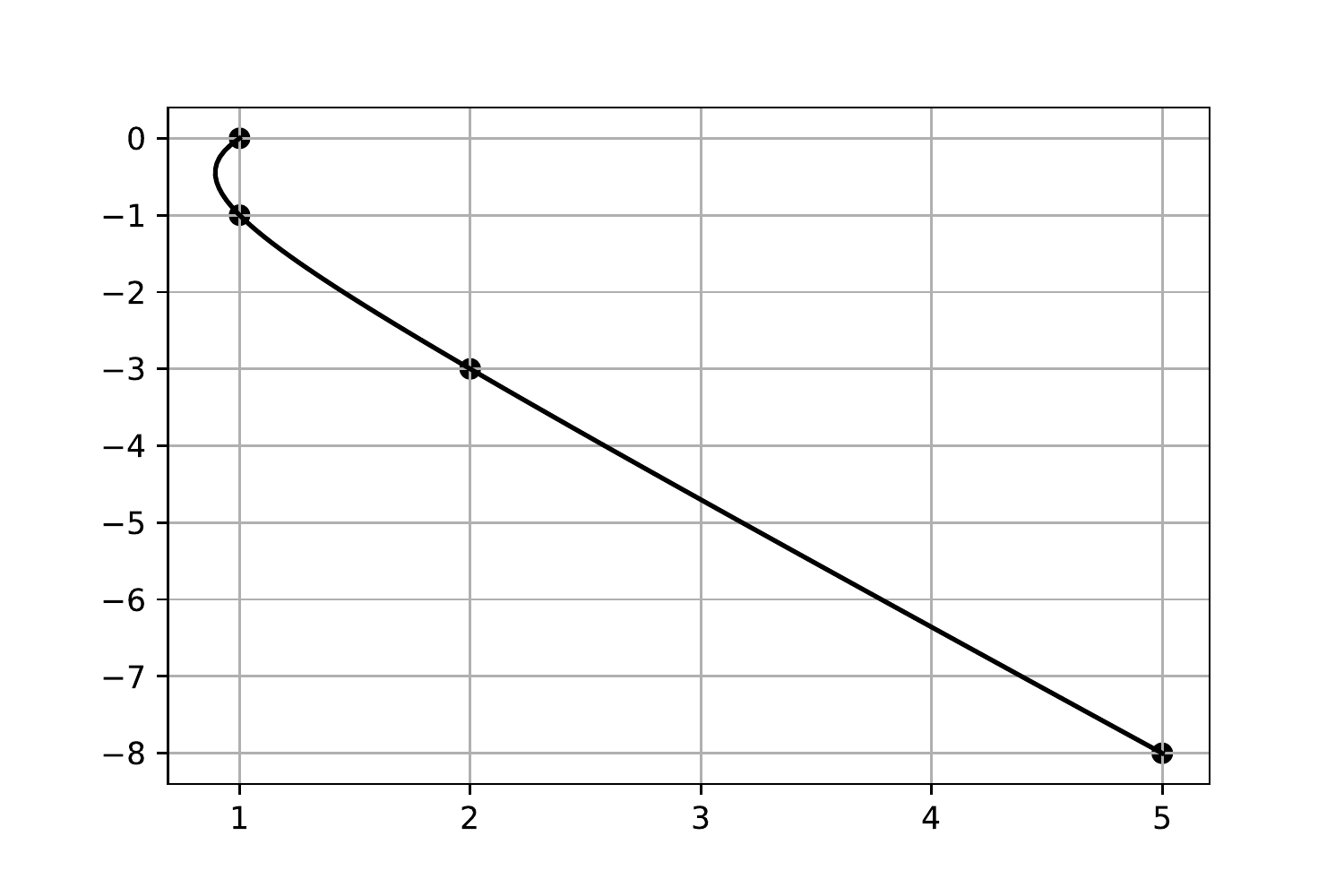}
\put(20,57){$p_1$}
\put(20,51){$p_2$}
\put(34,42){$p_3$}
\put(85,14){$p_4$}
\end{overpic}
\caption{The branch of $x^2-xy-y^2=1$ through $(1,0)$ showing
the lattice points on this curve with $y$ coordinate satisfying
$-8\le y\le 0$.  Here $p_1=(1,0)$, $p_2 =(1,-1)$,  $p_3 = (2,-3)$
and $p_4=(5,-8)$.}
\label{fig:hyp_example}
\end{figure}

%%%%%%%%%%%%%%%%%%%%%%%%%%%%%%%%%%%%%%%%%%%%%%%%%%%%%%%%%%%%%%%%%%%%%%
\subsection{Examples when $k_0>0$.}
\label{subsec:k>0}
In looking for examples our results are sharp when the lower
curvature bound, $k_0$, is positive we need to find evenly
spaced points on a curve, $\C$, with constant curvature
$k_0$.  Such a curve is a ellipse and after an affine 
motion we may assume it is a circle centered at the origin
with radius $r= k_0^{-3/2}$ and having unit affine speed
parametrization 
$$
c(s) = ( r\cos(k_0^{1/2}s), r \sin(k_0^{1/2}s)).
$$

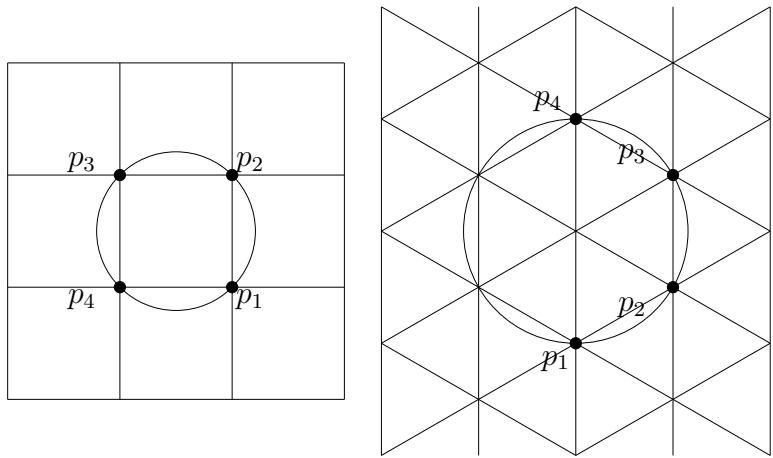
\begin{figure}[th]
\begin{overpic}[width=4in]{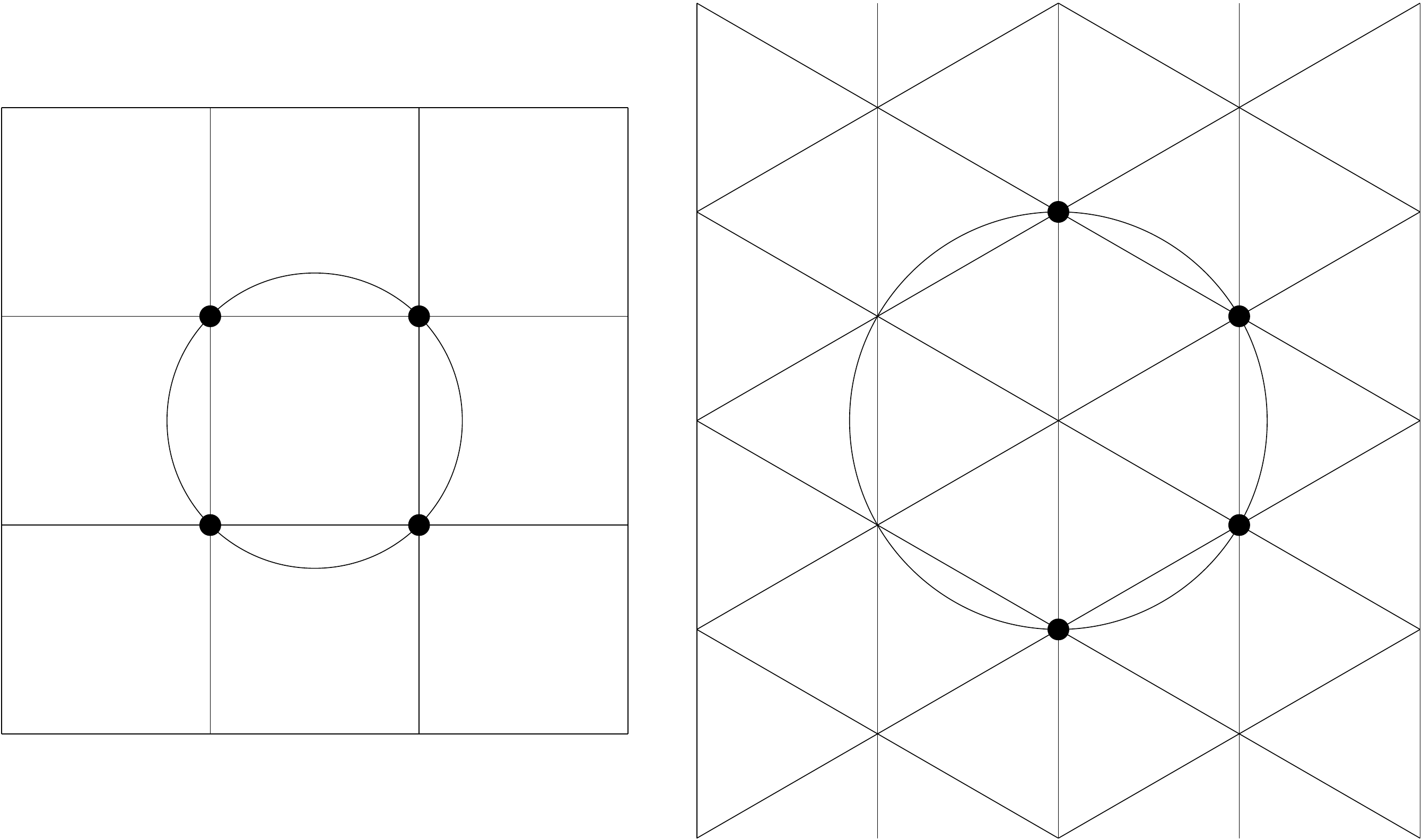}
\put(30,20){$p_1$}
\put(30,38){$p_2$}
\put(8,38){$p_3$}
\put(8,20){$p_4$}
\put(70,12){$p_1$}
\put(69,46){$p_4$}
\put(80,19){$p_2$}
\put(80,39){$p_3$}
\end{overpic}
\caption{In the case of a curve of positive
constant affine 
curvature up to an affine motion the only configurations
where there are four points satisfying from the lattice
on the curve  that they are evenly 
spaced with respect to affine arc length and
so that $\area(\triangle p_1p_2p_2)$ is half the
area of a fundamental region of the lattice
are shown here.
(In the second figure note that a fundamental region for the lattice
is not on of the triangles, but a parallelogram consisting of two
of the triangles.)}
\label{fig:circ_example}
\end{figure}

Let $\lat$ be a lattice so that there are four points
$p_1,p_2,p_3,p_4\in \lat \cap \C$ that are equally spaced
with respect to affine arc length and so that 
$H_{k_0}(L)= A_\lat/2$ where $L$ is the affine distance
between $p_j$ and $p_{j+1}$.  If equality holds in Theorem 
\ref{thm:rigid_lat} with $m\ge 2$ then four such points exist.
Proposition \ref{prop:equal_spaced} gives an affine motion
$\phi$ that preserves both $\lat$ and $\C$.  As $\C$ is
a circle centered at the origin this implies $\phi$ is
a rotation about the origin.  (The rotation being with
respect to the Euclidean structure that makes $\C$ into
a circle of radius $r$.)  If $\lat=\lat(v_0,v_1,v_2)$
and $\phi(v) = Mv$ (it is linear as it fixes the origin)
then the matrix of $M$ with respect to the basis $v_1$,
$v_2$ has integer entries and therefore its trace is
an integer.  As $M$ is a rotation, its matrix with respect
to the standard basis is $ \left[ \begin{matrix}
\cos(\theta)& -\sin(\theta)\\ \sin(\theta)& \cos(\theta)\end{matrix}\right]$ 
where $\theta$ is the angle of rotation.  The trace of
this is $2\cos(\theta)$ and therefore $2\cos(\theta)$
is an integer.  This implies $\theta$ is either an integral 
multiple of either  $\pi/3$ or $\pi/2$.  The only lattices
where we can get four points equally space along the circle
and so that $\area(\triangle p_2p_3p_4)$ is half the
area of a fundamental region of the lattice are shown
in Figure \ref{fig:circ_example}.  Therefore
Theorems \ref{thm:sharp_lat} and \ref{thm:rigid_lat} can only be
sharp 
when $\#(\lat \cap \C)$ is small to be precise $\#(\lat \cap \C)\le 6$.

\section*{Acknowledgments}
This paper is an outgrowth of work with Oggie  Trionov 
and conversions with him motivated me to apply affine
geometry to lattice point estimates.  I had useful 
conversations/correspondence with Dan Dix and
Grant Gustafson related to the results in Section
\ref{sec:initial_valus_problems}.  Much of the
this work was done while the author was on sabbatical 
leave from the University of South Carolina.

% \bibliographystyle{abbrv}
% \bibliography{affine_refs}

\end{document}